\documentclass[a4paper]{article}

\usepackage[dvipdfmx]{graphicx}
\usepackage{amsmath,amssymb,amsthm}
\usepackage{mathtools}
\usepackage{thmtools}

\usepackage{hyperref}
\usepackage[capitalize]{cleveref}
\usepackage{array}
\usepackage[all]{xy}
\usepackage{enumitem}
\usepackage{comment}
\usepackage{accents}
\usepackage{color} 
\usepackage{bbm}

\usepackage{mymacro}

\declaretheorem[numberwithin=section]{theorem}
\declaretheorem[sibling=theorem]{lemma}
\declaretheorem[sibling=theorem]{proposition}

\declaretheorem[style=definition,numbered=no]{definition}

\declaretheorem[style=remark,sibling=theorem]{example}
\declaretheorem[style=remark,sibling=theorem]{remark}
\declaretheorem[style=remark,numbered=no]{notation}

\numberwithin{equation}{section}
\numberwithin{figure}{section}

\DeclareMathAccent{\mss}{\mathrel}{letters}{"5E}

\title{Group operads as crossed interval groups}
\date{\today}
\author{Jun Yoshida}

\begin{document}
\maketitle


\tableofcontents

\section*{Introduction}
\label{sec:intro}
\addcontentsline{toc}{section}{Introduction}

The notion of group operads was developed in order to present various symmetries on operads.
The original definition was proposed by Zhang in his paper \cite{Zha11} though the axioms were already stated in 1.2.0.2 in the paper \cite{Wah01}.
Further theories in a more categorical viewpoint were developed by Corner and Gurski in the paper \cite{CG13} with the different name ``action operads.''
We here explain the fundamental idea.
Recall first that there are classically two conventions for ``operads''; namely, planar ones and symmetric ones.
For an operad $\mathcal O$, the latter also cares about a right action of the symmetric group $\mathfrak S_n$ on the set $\mathcal O(n)$ which is subject to a certain compatibility condition with the operad structure of $\mathcal O$.
The key observation is that the family $\{\mathfrak S_n\}_n$ itself forms an operad $\mathfrak S$ with $\mathfrak S(n)=\mathfrak S_n$.
For a symmetric operad $\mathcal O$, the compatibility condition is described as below:
\[
\gamma_{\mathcal O}(x^\sigma;x_1^{\sigma_1},\dots,x_n^{\sigma_n})
= \gamma_{\mathcal O}(x;x_{\sigma^{-1}(1)},\dots,x_{\sigma^{-1}(n)})^{\gamma_{\mathfrak S}(\sigma;\sigma_1,\dots,\sigma_n)}
\]
for $x\in\mathcal O(n)$, $x_i\in\mathcal O(k_i)$, $\sigma\in\mathfrak S_n$, and $\sigma_i\in\mathfrak S_{k_i}$, where $\gamma_{\mathcal O}$ and $\gamma_{\mathfrak S}$ are the composition operations in the operad $\mathcal O$ and $\mathfrak S$ respectively.
A group operad is, roughly, an operad $\mathcal G$ with each $\mathcal G(n)$ equipped with a group structure so that we can put it in place of the operad $\mathfrak S$ in the above argument.
The examples include the operad of braid groups, of pure braid groups, and of ribbon braid groups.

On the other hand, there is another approach to symmetries on operads.
Namely, Batanin and Markl \cite{BataninMarkl2014} focused on the notion of crossed interval groups, which are ``interval version'' of crossed simplicial groups introduced in \cite{FL91} and \cite{Kra87}.
They showed that the family $\mathfrak S=\{\mathfrak S_n\}_n$ also admits a structure of a crossed interval group and that the ``free symmetrization'' by $\mathfrak S$ does not change the homotopy type of operads of chain complexes.
They used this fact to prove that the operad of all natural operations on the Hochschild cochain complex of associative algebras has the same homotopy type of the singular chain of the little discs operad.
As for the general theory of crossed interval groups, the author investigated them in the paper \cite{Yoshidalimcolim}, where a classification result was obtained.
Namely, the terminal object $\mathfrak W_\nabla$ in the category of crossed interval groups has exactly six crossed interval subgroups $\ast$, $C_2$, $\mathfrak S$, $\mathfrak S\times C_2$, $\mathfrak H$, and $\mathfrak W_\nabla$.

The goal of this paper is to establish a comparison result of the two notions above.
We will see that group operads are nothing but a special kind of crossed interval groups.
More precisely, we will construct a functor $\widehat\Psi$ from the category $\mathbf{GrpOp}$ of group operads to the category $\mathbf{CrsGrp}_\nabla$ of crossed interval groups.
This functor is fully faithful, and there is an explicit description for the essential image of it.
The result is established in the following way.

In \cref{sec:grpoper}, the basic definition and examples of group operads will be reviewed.
Moreover, we will see that the slice category $\mathbf{Op}^{/\mathfrak S}$ of the category of operads over $\mathfrak S$ admits a monoidal structure $\rtimes$.
For a group operad $\mathcal G$, the group structure on each $\mathcal G(n)$ gives rise to a structure
\[
\mathcal G\rtimes\mathcal G\to\mathcal G
\ ,\quad \ast\to\mathcal G
\]
of a monoid object.
This defines a fully faithful embedding
\[
\mathbf{GrpOp}\to\mathbf{Mon}(\mathbf{Op}^{/\mathfrak S},\rtimes)
\]
into the category of monoid objects.
The necessary and sufficient condition for a monoid object to produce a group operad will be given.

On the other hand, as reviewed in \cref{sec:crsintgrp}, a crossed interval group is also a monoid object in the slice category $\mathbf{Set}^{/\mathfrak W_\nabla}_\nabla$ over $\mathfrak W_\nabla$ of the category of presheaves on the category $\nabla$ of intervals with respect to a monoidal structure $\rtimes$.
The essential image of $\mathbf{CrsGrp}_\nabla\to\mathbf{Mon}(\mathbf{Set}^{/\mathfrak W_\nabla}_\nabla,\rtimes)$ is described in a similar way to the case of group operads.
Note that the monoidal structure $\rtimes$ is monoidally closed so that the category $\mathbf{Mon}(\mathbf{Set}^{/\mathfrak W_\nabla}_\nabla,\rtimes)$ is locally presentable.
Since $\mathbf{CrsGrp}_\nabla$ is a reflective and coreflective subcategory of it, $\mathbf{CrsGrp}_\nabla$ is also locally presentable.

The two monoidal structures above are denoted by the same symbol because they are actually \emph{the same}.
Indeed, in \cref{sec:emb-ptdop}, we will construct a functor $\Psi:(\mathbf{Op}^{/\mathfrak S})^{\ast/}\to\mathbf{Set}^{/\mathfrak S}_\nabla$ and show it is strictly monoidal.
This implies that it induces a functor between the categories of monoid objects.
The computations of the essential images of $\mathbf{GrpOp}$ and of $\mathbf{CrsGrp}_\nabla$ in the category of monoid objects shows the functor actually restricts to $\mathbf{GrpOp}\to\mathbf{CrsGrp}^{/\mathfrak S}_\nabla$, which is exactly $\widehat\Psi$ mentioned above.
It will be proved that $\widehat\Psi$ is fully faithful.
In addition, a version of Adjoint Functor Theorem implies $\widehat\Psi$ admits a left adjoint, which enables us to regard $\mathbf{GrpOp}$ as a reflective subcategory of $\mathbf{CrsGrp}^{/\mathfrak S}_\nabla$.

In \cref{sec:operintgrp}, we will determine the essential image of the embedding $\widehat\Psi$.
The two conditions on crossed interval groups over $\mathfrak S$ will be considered.
One is the \emph{operadicity} which is the commutativity of certain elements.
If $G$ is an operadic crossed interval group, then it gives rise to an operad $\mathcal O_G$.
This implies operadic crossed interval group live in between group operads and operads with level-wise group structures.
Another condition is \emph{tameness}.
It will turn out that a group operad is exactly a tame and operadic crossed interval group over $\mathfrak S$.
Since there are ``operadification'' and ``taming'' of crossed interval group over $\mathfrak S$, this will give an explicit description of the left adjoint to the embedding $\widehat\Psi$.

We finally mention that, combining the embedding $\mathbf{GrpOp}\hookrightarrow\mathbf{CrsGrp}^{/\mathfrak S}_\nabla$ with the base-change theorem for crossed groups along the functor $\mathfrak J:\widetilde\Delta\to\nabla$ discussed in \cite{Yoshidalimcolim}, we obtain an augmented crossed simplicial group $\mathfrak J^\natural\mathcal G$ for each group operad $\mathcal G$.
We will show its \emph{total category} $\widetilde\Delta_{\mathcal G}$ classifies monoid objects in monoidal categories.
Note that this is an extension of the result of MacLane \cite{McL98}.
Actually, this observation clarify how the concept of Hochschild homologies for algebras makes sense in $\mathcal G$-symmetric monoidal abelian categories for general group operad $\mathcal G$.

\subsection*{Acknowledgment}

I would like first to thank my supervisor Prof. Toshitake Kohno for the encouragement and a lot of kind support.
Some important ideas were obtained after discussions with Prof. Ross Street.
This work was supported by the Program for Leading Graduate Schools, MEXT, Japan.
This work was supported by JSPS KAKENHI Grant Number JP15J07641.

\section{Group operads}
\label{sec:grpoper}

In this first section, we recall the formal definition of group operads.
Note that, throughout the paper, we use the convention where every operads are planar, unenriched (or $\mathbf{Set}$-enriched), and single colored.

\begin{notation}
\begin{enumerate}[label={\rm(\arabic*)}]
  \item For each natural number $n\in\mathbb N$, write
\[
\langle n\rangle:=\{1,\dots,n\}\ .
\]
We often regard it as the linearly ordered set with the canonical order.
  \item If $P$ and $Q$ are poset, i.e. partially ordered set, then we denote by $P\star Q$ the join of them.
In other words, $P\star Q$ is the set $P\amalg Q$ together with the ordering so that for $x,y\in P\star Q$,
\[
x\le y
\iff
\begin{cases}
(x,y)\in (P\times P)\amalg (Q\times Q)\  \text{with}\  x\le y
\ \text{, or} \\
x\in P\ \text{and}\  y\in Q\ .
\end{cases}
\]
Hence, we have a unique isomorphism $\langle m\rangle\star\langle n\rangle\cong\langle m+n\rangle$ of posets.
\end{enumerate}
\end{notation}

For each natural number $n\in\mathbb N$, we put $\mathfrak S(n)$ the $n$-th permutation group, or the permutation group on the set $\langle n\rangle=\{1,\dots,n\}$.
We begin with the observation that the family $\mathfrak S=\{\mathfrak S(n)\}_n$ admits a canonical structure of operads.
Namely, if $\sigma\in\mathfrak S(n)$ and $\sigma_i\in\mathfrak S(k_i)$ for $1\le i\le n$, we write $\gamma(\sigma;\sigma_1,\dots,\sigma_n)$ the permutation on $\langle k_1+\dots+k_n\rangle$ given below:
\[
\begin{split}
\langle k_1+\dots+k_n\rangle
&\cong \langle k_1\rangle\star\dots\star\langle k_n\rangle \\
&\xrightarrow{\sigma_1\amalg\dots\amalg\sigma_n} \langle k_1\rangle\star\dots\star\langle k_n\rangle \\
&\xrightarrow{\sigma_\ast} \langle k_{\sigma^{-1}(1)}\rangle\star\dots\star\langle k_{\sigma^{-1}(n)}\rangle \\
&\cong\langle k_1+\dots+k_n\rangle\ .
\end{split}
\]
This defines a map
\[
\gamma=\gamma_{\mathfrak S}:\mathfrak S(n)\times\prod_{i=1}^n\mathfrak S(k_1)\to\mathfrak S(k_1+\dots+k_n)\ .
\]
It is tedious but not difficult to see it makes $\mathfrak S$ into an operad.

Group operads are likely ``generalizations'' of the operad $\mathfrak S$.
For the definition, we follow \cite{CG13} for conventions except the terminology.

\begin{definition}
A \emph{group operad} is an operad $\mathcal G$ together with data
\begin{itemize}
  \item a group structure on each $\mathcal G(n)$;
  \item a map $\mathcal G\to\mathfrak S$ of operads so that each $\mathcal G(n)\to\mathfrak S(n)$ is a group homomorphism, which gives rise to a left $\mathcal G(n)$-action on $\langle n\rangle$;
\end{itemize}
which satisfy the identity
\begin{equation}
\label{eq:grpop-interchange}
\gamma_{\mathcal G}(xy;x_1y_1,\dots,x_ny_n)
= \gamma_{\mathcal G}(x;x_{y^{-1}(1)},\dots,x_{y^{-1}(n)})\gamma_{\mathcal G}(y;y_1,\dots,y_n)
\end{equation}
for every $x,y\in\mathcal G(n)$ and $x_i,y_i\in\mathcal G(k_i)$ for $1\le i\le n$.
\end{definition}

\begin{remark}
Note that, in the paper \cite{CG13} the terminology ``action operads'' was used.
This is probably because the name ``group operads'' may be confusing with \emph{group-enriched operads} or \emph{group objects in operads}.
Nevertheless, we stick to the terminology in a certain reason, which will turn out later.
\end{remark}

\begin{example}
The operad $\mathfrak S$ is a group operad with the identity map $\mathfrak S\xrightarrow=\mathfrak S$.
\end{example}

\begin{example}
For each $n\in\mathbb N$, denote by $\mathcal B(n)$ the braid group of $n$-strands.
In a similar manner to $\mathfrak S$, one can endow the family $\mathcal B=\{\mathcal B(n)\}_n$ with the structure of operads.
Then, the canonical quotient map $\mathcal B\to\mathfrak S$ exhibits $\mathcal B$ as a group operad.
The similar argument works for pure braids, ribbon braids, and so on.
\end{example}

The reader can find more interesting examples in \cite{Zha11} and \cite{Gur15}.
We here mention basic properties of group operads.

\begin{proposition}
\label{prop:grpop-unit}
For a group operad $\mathcal G$, the following hold.
\begin{enumerate}[label={\rm(\arabic*)}]
  \item\label{req:grpop-unit:unit} The composition map
\[
\gamma:\mathcal G(1)\times\mathcal G(1)\to\mathcal G(1)
\]
in the operad structure on $\mathcal G$ coincides with the multiplication in the group structure.
In particular, the unit $e_1\in\mathcal G(1)$ in the group structure is exactly the identity of the operad $\mathcal G$.
Moreover, $\mathcal G(1)$ is an abelian group.
  \item\label{req:grpop-unit:comp} For each $n\in\mathbb N$, write $e_n\in\mathcal G(n)$ the unit in the group structure.
Then, for $k_1,\dots,k_n\in\mathbb N$, we have
\[
\gamma(e_n;e_{k_1},\dots,e_{k_n})=e^{}_{k_1+\dots+k_n}\ .
\]
In other words, the family $\{e_n\}_n$ determines a map $\ast\to\mathcal G$ of operads from the terminal (or trivial) operad $\ast$.
  \item\label{req:grpop-unit:homo} For each $n\in\mathbb N$, the map
\[
\mathcal G(1)\to\mathcal G(n)
\ ;\quad x\mapsto \gamma(x;e_n)
\]
is a group homomorphism.
\end{enumerate}
\end{proposition}
\begin{proof}
Notice that, for $x,x',y,y'\in\mathcal G(1)$, the condition \eqref{eq:grpop-interchange} on group operads implies
\begin{equation}
\label{eq:prf:grpop-unit:EHarg}
\gamma(xx';yy')=\gamma(x;y)\gamma(x';y')\ .
\end{equation}
Hence, the part \ref{req:grpop-unit:unit} follows from the \emph{Eckmann-Hilton argument}.

To see \ref{req:grpop-unit:comp}, since $\mathcal G(k_1+\dots+k_n)$ is a group, it suffices to see the element $\gamma(e_n;e_{k_1},\dots,e_{k_n})$ is idempotent.
By the condition on group operads again, we have
\[
\gamma(e_n;e_{k_1},\dots,e_{k_n})^2
=\gamma(e_n^2;e_{k_1}^2,\dots,e_{k_n}^2)
=\gamma(e_n;e_{k_1},\dots,e_{k_n})\ .
\]
This implies $\gamma(e_n;e_{k_1},\dots,e_{k_n})=e_{k_1+\dots+k_n}$.

The last assertion \ref{req:grpop-unit:homo} directly follows from the condition \eqref{eq:grpop-interchange} and the part \ref{req:grpop-unit:unit}.
\end{proof}

\begin{definition}
Let $\mathcal G$ and $\mathcal H$ be group operads.
Then, a \emph{map of group operads} is a map $f:\mathcal G\to\mathcal H$ of operads such that
\begin{enumerate}[label={\rm(\roman*)}]
  \item for each $n\in\mathbb N$, the map $f:\mathcal G(n)\to\mathcal H(n)$ is a group homomorphism;
  \item the triangle below is commutative:
\[
\vcenter{
  \xymatrix@C=1em{
    \mathcal G \ar[rr]^f \ar[dr] && \mathcal H \ar[dl] \\
    & \mathfrak S & }}\ ,
\]
where the vertical arrows are the structure maps.
\end{enumerate}
\end{definition}

Clearly maps of group operads compose so as to form a category, which we will denote by $\mathbf{GrpOp}$.
One of the important results proved in \cite{Gur15} is the presentability of the category.

\begin{theorem}[Theorem~3.5 and Theorem~3.8 in \cite{Gur15}]
\label{theo:grpop-locpres}
The category $\mathbf{GrpOp}$ is locally finitely presentable.
Moreover, the forgetful functor $U:\mathbf{GrpOp}\to\mathbf{Set}_{\mathbb N}^{/\mathfrak S}$ into the slice category of $\mathbb N$-indexed families of sets over the family $\{\mathfrak S(n)\}_n$ creates limits and filtered colimits.
\end{theorem}

We next see the notion of group operads also arise from a monoidal structure.

\begin{definition}
Let $\mathcal X$ and $\mathcal Y$ be two operads, and suppose we are given a map $\rho:\mathcal Y\to\mathfrak S$ of operads.
We define an operad $\mathcal X\rtimes\mathcal Y$ as follows:
\begin{itemize}
  \item for each $n\in\mathbb N$, put $(\mathcal X\rtimes\mathcal Y)(n):=\mathcal X(n)\times\mathcal Y(n)$;
  \item for $(x,y)\in(\mathcal X\rtimes\mathcal Y)(n)$ and $(x_i,y_i)\in(\mathcal X\rtimes\mathcal Y)(k_i)$ for $1\le i\le n$, the composition is given by
\[
\begin{multlined}
\gamma_{\mathcal X\rtimes\mathcal Y}((x,y);(x_1,y_1),\dots,(x_n,y_n)) \\
:= (\gamma_{\mathcal X}(x;x_{\rho(y)^{-1}(1)},\dots,x_{\rho(y)^{-1}(n)}),\gamma_{\mathcal Y}(y;y_1,\dots,y_n))\ ;
\end{multlined}
\]
\end{itemize}
\end{definition}

It is easily verified that the above data actually define an operad $\mathcal X\rtimes\mathcal Y$ so that the identity is the pair $\mathrm{id}_{\mathcal X\rtimes\mathcal Y}=(\mathrm{id}_{\mathcal X},\mathrm{id}_{\mathcal Y})$.
Moreover, the assignment $(\mathcal X,\mathcal Y,\rho)\mapsto \mathcal X\rtimes\mathcal Y$ is functorial; indeed, if we have a map $f:\mathcal X\to\mathcal X'$ and a triangle
\[
\xymatrix@C=1em{
  \mathcal Y \ar[rr]^g \ar[dr]_\rho && \mathcal Y' \ar[dl]^{\rho'} \\
  & \mathfrak S & }
\]
of operads, the maps
\[
f\rtimes g:(\mathcal X\rtimes\mathcal Y)(n)
\to (\mathcal X'\rtimes\mathcal Y')(n)
\ ;\quad (x,y) \mapsto (f(x),g(y))
\]
form a map $\mathcal X\rtimes\mathcal Y\to\mathcal X'\rtimes\mathcal Y'$ of operads.
In other words, if we denote by $\mathbf{Op}$ the category of operads and by $\mathbf{Op}^{/\mathfrak S}$ the slice category over $\mathfrak S$, then we obtain a functor
\begin{equation}
\label{eq:rtimes-OpxOpS}
\rtimes:\mathbf{Op}\times\mathbf{Op}^{/\mathfrak S}\to\mathbf{Op}\ .
\end{equation}

The following result is a direct consequence of the operad structure of $\mathfrak S$.

\begin{lemma}
\label{lem:operS-prod}
The multiplication maps
\[
\mathtt{mul}:\mathfrak S(n)\times\mathfrak S(n)\to\mathfrak S(n)
\ ;\quad (\sigma,\tau)\mapsto \sigma\tau
\]
define a map $\mathfrak S\rtimes\mathfrak S\to\mathfrak S$ of operads, here we take $\mathfrak S\rtimes\mathfrak S$ with respect to the identity map $\mathfrak S\xrightarrow=\mathfrak S$.
\end{lemma}

\Cref{lem:operS-prod} offers a lift of the functor \eqref{eq:rtimes-OpxOpS} to a \emph{binary operation} on $\mathbf{Op}^{/\mathfrak S}$; indeed, we have the following composition
\begin{equation}
\label{eq:rtimes-OpSxOpS}
\rtimes:\mathbf{Op}^{/\mathfrak S}\times\mathbf{Op}^{/\mathfrak S}
\cong (\mathbf{Op}\times\mathbf{Op}^{/\mathfrak S})^{/(\mathfrak S,\mathfrak S)}
\xrightarrow{\rtimes} \mathbf{Op}^{/\mathfrak S\rtimes\mathfrak S}
\xrightarrow{\mathtt{mul}_!} \mathbf{Op}^{/\mathfrak S}\ .
\end{equation}

\begin{proposition}
\label{prop:rtimes-monstr}
The functor \eqref{eq:rtimes-OpSxOpS} defines a monoidal structure on the category $\mathbf{Op}^{/\mathfrak S}$ so that the trivial operad $\ast$ is the unit object.
\end{proposition}
\begin{proof}
Since the last statement is obvious, we have to give an associativity isomorphism.
It suffices to show that, for operads $\mathcal X$, $\mathcal Y$, and $\mathcal Z$ over $\mathfrak S$, the identification
\[
((\mathcal X\rtimes\mathcal Y)\rtimes\mathcal Z)(n)
= \mathcal X(n)\times\mathcal Y(n)\times\mathcal Z(n)
= (\mathcal X\rtimes(\mathcal Y\rtimes\mathcal Z))(n)
\]
is an isomorphism $(\mathcal X\rtimes\mathcal Y)\rtimes\mathcal Z\cong\mathcal X\rtimes(\mathcal Y\rtimes\mathcal Z)$.
Actually, in either case, the composition operation is given by
\[
\begin{multlined}
\gamma((x,y,z);(x_1,y_1,z_1),\dots,(x_n,y_n,z_n)) \\
= \bigl(\gamma_{\mathcal X}(x;x_{\pi(z)^{-1}\rho(y)^{-1}(1)},\dots,x_{\pi(z)^{-1}\rho(y)^{-1}(n)}),
\gamma_{\mathcal Y}(y;y_{\pi(z)^{-1}(1)},\dots,y_{\pi(z)^{-1}(n)}), \\
\gamma_{\mathcal Z}(z;z_1,\dots,z_n)\bigr)\ ,
\end{multlined}
\]
where $\rho:\mathcal Y\to\mathfrak S$ and $\pi:\mathcal Z\to\mathfrak S$ are the structure map.
\end{proof}

\begin{definition}
A \emph{monoid operad} is a monoid object in the category $\mathbf{Op}^{/\mathfrak S}$ with respect to the monoidal structure $\rtimes$.
\end{definition}

We denote by $\mathbf{MonOp}$ the category of monoid operads and monoid homomorphisms in $\mathbf{Op}^{/\mathfrak S}$.
Note that a monoid operad $\mathcal X$ consists of an operad together with data
\begin{itemize}
  \item a monoid structure on each $\mathcal X(n)$, and
  \item a map $\mathcal X\to\mathfrak S$ of operads so that $\mathcal X(n)\to\mathfrak S$ is a monoid homomorphism;
\end{itemize}
which satisfy appropriate conditions.
Comparing it with the definition of group operads, one may notice that a group operad $\mathcal G$ determines a monoid operad and that it gives rise to a functor $\mathbf{GrpOp}\to\mathbf{MonOp}$.
The following result is an easy exercise.

\begin{proposition}
\label{prop:grpop->monop}
The functor $\mathbf{GrpOp}\to\mathbf{MonOp}$ is fully faithful.
Moreover, a monoid operad $\mathcal X$ belongs to the essential image if and only if for each $n\in\mathbb N$, the monoid $\mathcal X(n)$ is a group.
\end{proposition}

One of the important features of monoid operads is their actions on multicategories.
The argument begin with the observation that the functor $\rtimes:\mathbf{Op}\times\mathbf{Op}^{/\mathfrak S}\to\mathbf{Op}$ extends to a functor
\begin{equation}
\label{eq:multcat-rtimes}
\rtimes:\mathbf{MultCat}\times\mathbf{Op}^{/\mathfrak S}\to\mathbf{MultCat},
\end{equation}
where $\mathbf{MultCat}$ is the category of (small) multicategories and multifunctors; as for the theory of multicategories, we refer the reader to \cite{Hermida2000} and \cite{Lei04}.
Indeed, for a multicategory $\mathcal M$ and an operad $\rho:\mathcal X\to\mathfrak S$ over $\mathfrak S$, we define a multicategory $\mathcal M\rtimes\mathcal X$ as follows:
\begin{itemize}
  \item objects are those of $\mathcal M$;
  \item for $a_1,\dots,a_n,a\in\mathcal M$, we set
\[
\begin{multlined}
\mathcal M(a_1\dots a_n;a) \\[1ex]
:= \left\{(f,x)\;\middle|\;x\in\mathcal X(n),\, f\in\mathcal M(a_{\rho(x)^{-1}(1)}\dots a_{\rho(x)^{-1}(n)};a)\right\};
\end{multlined}
\]
  \item the composition operation is defined so that
\[
\begin{split}
&\gamma_{\mathcal M\rtimes\mathcal X}((f,x);(f_1,x_1),\dots,(f_n,x_n)) \\
&:= \left(\gamma_{\mathcal M}(f;f_{\rho(x)^{-1}(1)},\dots,f_{\rho(x)^{-1}(n)}),\gamma_{\mathcal X}(x;x_1,\dots,x_n)\right)\ .
\end{split}
\]
\end{itemize}
It is easily checked that the composition makes sense and is associative.
Note that we have a canonical identification
\[
(\mathcal M\rtimes\mathcal X)\rtimes\mathcal Y
=\mathcal M\rtimes(\mathcal X\rtimes\mathcal Y)\ .
\]
Hence, the functor \eqref{eq:multcat-rtimes} exhibits the category $\mathbf{MultCat}$ as a right $\mathbf{Op}^{/\mathfrak S}$-module with respect to the monoidal structure $\rtimes$ on $\mathbf{Op}^{/\mathfrak S}$.
As a consequence, if $\mathcal X$ is a monoid operad, it gives rise to a functor
\[
(\blank)\rtimes\mathcal X:\mathbf{MultCat}\to\mathbf{MultCat}\ .
\]
It is actually a $2$-functor: it sends a multinatural transformation $\alpha:F\to G:\mathcal M\to\mathcal N$ to the multinatural transformation consisting of
\begin{equation}
\label{eq:rtX-mulnat}
(\alpha_a,e_1)
\in(\mathcal N\rtimes\mathcal G)(F(a);G(a))
= \mathcal N(F(a);G(a))\times\mathcal G(1)
\end{equation}
for each $a\in\mathcal M$.
Furthermore, it is easily verified that the monoid operad structure on $\mathcal X$ makes the $2$-functor $(\blank)\rtimes\mathcal X$ into a $2$-monad: we have an obvious $2$-natural isomorphism
\[
((\blank)\rtimes\mathcal X)\rtimes\mathcal X
\cong(\blank)\rtimes(\mathcal X\rtimes\mathcal X)
\]
so that we can define two $2$-natural transformations
\[
\begin{gathered}
((\blank)\rtimes\mathcal X)\rtimes\mathcal X
\cong (\blank)\rtimes(\mathcal X\rtimes\mathcal X)
\xrightarrow{\mathrm{Id}\times\mathtt{mul}} (\blank)\rtimes\mathcal X\ ,
\\
\mathrm{Id}
\cong(\blank)\rtimes\ast
\xrightarrow{\mathrm{Id}\times\mathtt{unit}} (\blank)\rtimes\mathcal X\ .
\end{gathered}
\]

\begin{definition}
Let $\mathcal X$ be a monoid operad.
Then, an \emph{$\mathcal X$-symmetric structure} on a multicategory $\mathcal M$ is nothing but a structure of a strict $2$-algebra over the $2$-monad $(\blank)\rtimes\mathcal X$; i.e. a multifunctor
\[
\mathtt{sym}:\mathcal M\rtimes\mathcal X\to\mathcal M
\]
which makes the following diagrams commute:
\begin{equation}
\label{eq:Xsym-cohdiags}
\vcenter{
  \xymatrix@C=4em{
    \mathcal M\rtimes\mathcal X\rtimes\mathcal X \ar[r]^-{\mathrm{Id}\rtimes\mathtt{mul}} \ar[d]_{\mathtt{sym}\rtimes\mathrm{Id}} & \mathcal M\rtimes\mathcal X \ar[d]^{\mathtt{sym}} \\
    \mathcal M\rtimes\mathcal X \ar[r]^{\mathtt{sym}} & \mathcal M }}
\quad,\quad
\vcenter{
  \xymatrix@C=1em{
    \mathcal M\rtimes\ast \ar@{=}[dr] \ar[rr]^{\mathrm{Id}\times\mathtt{unit}} && \mathcal M\rtimes\mathcal X \ar[dl]^{\mathtt{sym}} \\
    & \mathcal M & }}
\quad.
\end{equation}
We say $\mathcal M$ is \emph{$\mathcal X$-symmetric} if it is equipped with an $\mathcal X$-symmetric structure.
\end{definition}

\begin{remark}
In the case $\mathcal X=\mathcal G$ is a group operad, the notion of $\mathcal G$-symmetric multicategory defined above is equivalent to that of $\mathcal G$-multicategories in Definition~5.1 in \cite{Gur15}.
\end{remark}

\begin{example}
If $\mathcal X=\ast$ is the trivial group operad, then $\ast$-symmetric multicategories are just (planer) multicategories.
\end{example}

\begin{example}
In the case $\mathcal X=\mathfrak S$ is the group operad of symmetric groups, $\mathfrak S$-symmetric multicategories are precisely \emph{symmetric multicategories} in the usual sense.
\end{example}

We also consider multifunctors respecting symmetries.

\begin{definition}
Let $\mathcal X$ be a monoid operad, and let $\mathcal M$ and $\mathcal N$ be $\mathcal X$-symmetric multicategory.
Then, a multifunctor $F:\mathcal M\to\mathcal N$ is said to be \emph{$\mathcal X$-symmetric} if it is a morphism of algebras over the $2$-monad $(\blank)\rtimes\mathcal X$; i.e. the following diagram commutes:
\[
\vcenter{
  \xymatrix{
    \mathcal M\rtimes\mathcal X \ar[r]^{F\rtimes\mathrm{Id}} \ar[d]_{\mathtt{sym}} & \mathcal N\rtimes\mathcal X \ar[d]^{\mathtt{sym}} \\
    \mathcal M \ar[r]^F & \mathcal N }}
\quad.
\]
\end{definition}

We denote by $\mathbf{MultCat}_{\mathcal X}$ the $2$-category of $\mathcal X$-symmetric multicategories, $\mathcal X$-symmetric multifunctors, and transformations of morphisms.

Note that we do not provide any special terminologies for $2$-morphisms in $\mathbf{MultCat}_{\mathcal X}$ of in $\mathbf{MonCat}_{\mathcal X}$ because of the following result.

\begin{lemma}
\label{lem:symnat-ff}
Let $\mathcal X$ be a monoid operad.
Then, the forgetful $2$-functor
\[
\mathbf{MultCat}_{\mathcal X}
\to\mathbf{MultCat}
\]
is locally fully faithful.
\end{lemma}
\begin{proof}
Take two $\mathcal X$-symmetric multicategory $\mathcal M$ and $\mathcal N$.
Note that the category $\mathbf{MultCat}_{\mathcal X}(\mathcal M,\mathcal N)$ is obtained as the equalizer of the parallel functors
\[
\vcenter{
  \xymatrix@C=5em@1{
    \mathbf{MultCat}(\mathcal M,\mathcal N) \ar@<.15pc>[r]^-{\mathtt{sym}_\ast\circ(\blank)\rtimes\mathcal X} \ar@<-.15pc>[r]_-{\mathtt{sym}^\ast} & \mathbf{MultCat}(\mathcal M\rtimes\mathcal X,\mathcal N) }}
\quad.
\]
Thanks to \eqref{eq:rtX-mulnat} and \eqref{eq:Xsym-cohdiags}, both functors are identities on morphisms, so we get the result.
\end{proof}

\begin{remark}
We are mainly interested in the case $\mathcal X$ is a group operad.
In this case, it was proved in \cite{CG13} that $\mathbf{MultCat}_{\mathcal X}$ is biequivalent to the $2$-category of \emph{pseudo-algebras} over $(\blank)\rtimes\mathcal X$.
\end{remark}

We finally discuss the case a multicategory comes from a monoidal category.
Recall that if $\mathcal C$ is a monoidal category, then we have the associated multicategory $\mathcal C^\otimes$ with the same objects as $\mathcal C$ and
\[
\mathcal C^\otimes(X_1\dots X_n;X)
:= \mathcal C(X_1\otimes\dots\otimes X_n;X)
\]
for $X,X_1,\dots,X_n\in\mathcal C$.
We define an \emph{$\mathcal X$-symmetric structure on $\mathcal C$} to be that on $\mathcal C^\otimes$ and say $\mathcal C$ is $\mathcal X$-symmetric if $\mathcal C^\otimes$ is so.
In this case, for each $x\in\mathcal G(n)$ and for objects $X_1,\dots, X_n\in\mathcal C$, we set
\begin{equation}
\label{eq:Gsym-moncat}
\Theta^x_{X_1\dots X_n}:
X_1\otimes\dots\otimes X_n
\to X_{x^{-1}(1)}\otimes X_{x^{-1}(n)}
\end{equation}
to be the image of the pair $(\mathrm{id},x)$ by the map
\[
\begin{split}
&(\mathcal C^\otimes\rtimes\mathcal X)(X_1\dots X_n;X_{x^{-1}(1)}\otimes\dots\otimes X_{x^{-1}(n)}) \\
&\xrightarrow{\mathtt{sym}} \mathcal C^\otimes(X_1\dots X_n;X_{x^{-1}(1)}\otimes\dots\otimes X_{x^{-1}(n)}) \\
&= \mathcal C(X_1\otimes\dots\otimes X_n;X_{x^{-1}(1)}\otimes\dots\otimes X_{x^{-1}(n)})\ .
\end{split}
\]
It is verified that the family $\Theta^x:=\{\Theta^x_{X_1\dots X_n}\}_{X_1,\dots,X_n}$ forms a natural transformation such that $\Theta^{e_n}=\mathrm{id}$ and $\Theta^x\Theta^y=\Theta^{xy}$.
In particular, if $\mathcal X=\mathfrak S$ (resp. $\mathcal B$), this natural transformation $\Theta^x$ is nothing but the appropriate composition of the \emph{braidings} in the symmetric (resp. braided) structure.
Hence, (resp. $\mathcal B$-symmetric) monoidal categories are nothing but symmetric (resp. braided) monoidal categories.

Let us denote by $\mathbf{MonCat}$ the $2$-category of monoidal categories, monoidal functors, and monoidal natural transformations.
For a monoid operad $\mathcal X$, we define a $2$-category $\mathbf{MonCat}_{\mathcal X}$ to be the pullback
\[
\mathbf{MonCat}_{\mathcal X}
:= \mathbf{MonCat}\times_{\mathbf{MultCat}}\mathbf{MultCat}_{\mathcal X}
\]
and call its $1$-morphisms \emph{$\mathcal G$-symmetric monoidal functors}.

\section{Crossed interval groups}
\label{sec:crsintgrp}

We first recall the definition of crossed groups.

\begin{definition}
Let $\mathcal A$ be a small category.
Then, a crossed $\mathcal A$-group is a presheaf $G$ over $\mathcal A$ equipped with data
\begin{itemize}
  \item a group structure on $G(a)$ for each $a\in\mathcal A$;
  \item a left $G(a)$-action
\[
G(a)\times\mathcal A(b,a)\to\mathcal A(b,a)
\ ;\quad (x,\varphi)\mapsto\varphi^x
\]
for each $a,b\in\mathcal A$;
\end{itemize}
satisfying the following two conditions:
\begin{enumerate}[label={\rm(\roman*)}]
  \item for morphisms $\varphi:b\to a$ and $\psi:c\to b$ in $\mathcal A$, and for $x\in G(a)$,
\[
(\varphi\psi)^x = \varphi^x\psi^{\varphi^\ast(x)}\ ;
\]
  \item for a morphism $\varphi:b\to a$ and for $x,y\in G(a)$,
\[
\varphi^\ast(xy) = (\varphi^y)^\ast(x)\varphi^\ast(y)\ .
\]
\end{enumerate}
\end{definition}

\begin{definition}
Let $\mathcal A$ be a small category.
For two crossed $\mathcal A$-groups $G$ and $H$, a map of crossed $\mathcal A$-groups from $G$ to $H$, written $G\to H$, is a map $f:G\to H$ of presheaves over $\mathcal A$ such that
\begin{enumerate}[label={\rm(\roman*)}]
  \item the map $f:G(a)\to H(a)$ is a group homomorphism for each $a\in\mathcal A$;
  \item for each morphism $\varphi:b\to a\in\mathcal A$, and for each $x\in G(a)$, we have $\varphi^{f(x)}=\varphi^x$.
\end{enumerate}
\end{definition}

\begin{remark}
The notion of crossed groups was first considered in the simplicial case by Fiedorowicz, Loday \cite{FL91}, and Krasauskas \cite{Kra87}.
The original motivation is to generalize the cyclic homologies of algebras.
\end{remark}

\begin{remark}
\label{rem:totcat}
Some authors define crossed $\mathcal A$-groups as extensions of $\mathcal A$ with unique factorizations.
In fact, if $G$ is a crossed $\mathcal A$-group, then we can form a category $\mathcal A_G$ as follows:
\begin{itemize}
  \item objects are the same as $\mathcal A$;
  \item for each $a,b\in\mathcal A$, we set $\mathcal A_G(a,b):=\mathcal A(a,b)\times G(a)$;
  \item the composition is given by the formula
\[
(\varphi,x)\circ(\psi,y) := (\varphi\psi^x,\psi^\ast(x)y)\ .
\]
\end{itemize}
The two conditions on crossed $\mathcal A$-groups above are almost equivalent to saying that the composition in $\mathcal A_G$ is associative.
We call $\mathcal A_G$ the \emph{total category of $G$}.
\end{remark}

Crossed $\mathcal A$-groups and maps of them form a category $\mathbf{CrsGrp}_{\mathcal A}$.
Several categorical properties of $\mathbf{CrsGrp}_{\mathcal A}$ are investigated in \cite{Yoshidalimcolim}.
One of the most important results is the following.

\begin{theorem}[Theorem~2.4 in \cite{Yoshidalimcolim}]
\label{theo:crsgrp-locpres}
For every small category $\mathcal A$, the category $\mathbf{CrsGrp}_{\mathcal A}$ is locally presentable.
In particular, it admits a terminal object.
\end{theorem}

In this paper, we are particularly interested in the case $\mathcal A$ is the \emph{category $\nabla$ of intervals}, which is described as follows:
\begin{itemize}
  \item objects of $\nabla$ are linearly ordered set of the form
\[
\llangle n\rrangle := \{-\infty,1,\dots,n,\infty\}
\]
for $n\in\mathbb N$;
  \item morphisms are order-preserving maps $\varphi:\llangle m\rrangle\to\llangle n\rrangle$ with $\varphi(\infty)=\infty$ and $\varphi(-\infty)=-\infty$.
\end{itemize}
This case was first investigated by Batanin and Markl \cite{BataninMarkl2014}, and they call objects of $\mathbf{CrsGrp}_\nabla$ \emph{crossed interval groups} so we follow it.

\begin{notation}
As a special treatment in the interval case, for a crossed interval group $G$, we will write $G_n := G(\llangle n\rrangle)$.
Note that it might be confusing in some conventions; some authors use this notation for crossed \emph{simplicial} groups.
Unfortunately, the canonical functors connecting the categories of crossed interval groups and simplicial ones requires the shift of the canonical degrees (see Example~5.5 and 5.6 in \cite{Yoshidalimcolim}).
We will warn the reader when there is a danger of confusion.
\end{notation}

To understand the category $\nabla$, it is convenient to introduce the following notation: for a morphism $\varphi:\llangle m\rrangle\to\llangle n\rrangle\in\nabla$, and for $j\in\llangle n\rrangle$, we put
\begin{equation}
\label{eq:nablamorph-vect}
k^{(\varphi)}_j=
\begin{cases}
\#\varphi^{-1}\{j\}-1 & j=\pm\infty\ , \\
\hfil\#\varphi^{-1}\{j\}\hfil & 1\le j\le n\ .
\end{cases}
\end{equation}
Note that the assignment $\varphi\mapsto \vec k^{(\varphi)}=(k^{(\varphi)}_{-\infty},k^{(\varphi)}_1,\dots,k^{(\varphi)}_n,k^\varphi_\infty)$ gives rise to a bijection
\[
\nabla(\llangle m\rrangle,\llangle n\rrangle)
\cong \left\{\vec k=(k_{-\infty},k_1,\dots,k_n,k_\infty)\;\middle|\; \sum_{j\in\llangle n\rrangle}k_j=m\right\}\ .
\]
We make use of the notation above to introduce a left action of $\mathfrak S_n\times\mathbb Z/2\mathbb Z$ on the set $\nabla(\llangle m\rrangle,\llangle n\rrangle)$ as follows: for $(\sigma;\varepsilon)\in\mathfrak S_n\times\mathbb Z/2\mathbb Z$ and for $\varphi:\llangle m\rrangle\to\llangle n\rrangle\in\nabla$, define $\varphi^{(\sigma;\varepsilon)}:\llangle m\rrangle\to\llangle n\rrangle$ so that
\[
k^{(\varphi^{(\sigma;\varepsilon)})}_{\pm\infty}
= k^{(\varphi)}_{\pm(-1)^\varepsilon\infty}
\ \,\quad
k^{(\varphi^{(\sigma;\varepsilon)})}_j
= k^{(\varphi)}_{\sigma^{-1}(j)}
\]
Notice that, taking $m=1$, one gets a canonical identification $\nabla(\llangle1\rrangle,\llangle n\rrangle)=\llangle n\rrangle$.
In terms of the action above, the permutation $(\sigma;\varepsilon):\llangle n\rrangle\to\llangle n\rrangle$ sends $1\le j\le n$ to $\sigma(j)$ and $\pm\infty$ to $\pm(-1)^\varepsilon\infty$.

There is a typical recipe to construct crossed interval groups.
We begin with a commutative triangle of group homomorphisms below.
\begin{equation}
\label{eq:grptri/Z/2}
\vcenter{
  \xymatrix@C=1em{
    H_0 \ar[rr]^\theta \ar[dr] && H_1 \ar[dl]^{\varepsilon} \\
    & \mathbb Z/2\mathbb Z & }}
\end{equation}
By abuse of notation, we denote the triangle just by the pair $(H_0,H_1)$.
For each $n\in\mathbb N$, define a group $W(H_0,H_1)_n$ by
\[
W(H_0,H_1)_n:=(\mathfrak S_n\ltimes H_1^{\times n})\times H_0\ ,
\]
where $\mathfrak S_n$ acts on $H_1^{\times n}$ from the right as the permutations of indices.
Hence, for elements $(\sigma;x_1,\dots,x_n;u),(\tau;y_1,\dots,y_n;v)\in W(H_0,H_1)_n$, we have
\[
(\tau;y_1,\dots,y_n;v)\cdot(\sigma;x_1,\dots,x_n;u)
=(\tau\sigma;y_{\sigma(1)}x_1,\dots,y_{\sigma(n)}x_n;vu)\ .
\]
To introduce an interval set structure, we write $\beta_n\in\mathfrak S_n$ the order-reversion map on $\langle n\rangle$; i.e. $\beta_n(i)=n-i+1$.
Then, for a map $\varphi:\llangle m\rrangle\to\llangle n\rrangle$, define $\varphi^\ast:W(H_0,H_1)_n\to W(H_0,H_1)_m$ by
\begin{multline}
\label{eq:WH-interval}
\varphi^\ast(\sigma;x_1,\dots,x_n;u) \\
:= (\gamma(\beta_3^u;\beta^u_{k^{(\varphi)}_{-\infty}},\gamma(\sigma;\beta^{x_1}_{k^{(\varphi)}_1},\dots,\beta^{x_n}_{k^{(\varphi)}_n}),\beta^u_{k^{(\varphi)}_\infty});x_{\varphi(1)},\dots,x_{\varphi(m)};u)\ ,
\end{multline}
assuming $x_{-\infty}=x_{\infty}=\theta(u)$, where $\gamma$ is the composition operation in the operad structure of $\mathfrak S$, and $\beta^u_k$ and $\beta^{x_i}_k$ are abbreviations of $\beta^{\varepsilon\theta(u)}_k$ and $\beta^{\varepsilon(x_i)}_k$ respectively.
Fortunately, there is a more conceptual description for the complicated permutation appearing in \eqref{eq:WH-interval}; it is the permutation $\tilde\sigma$ on $\langle m\rangle$ so that the square
\[
\xymatrix@C=3em{
  \langle m\rangle \ar[d]_{\tilde\sigma} \ar@{^(->}[]+R+(1,0);[r] & \llangle m\rrangle \ar[r]^\varphi & \llangle n\rrangle \ar[d]^{(\sigma;u)} \\
  \langle m\rangle \ar@{^(->}[]+R+(1,0);[r] & \llangle m\rrangle \ar[r]^{\varphi^{(\sigma;u)}} & \llangle n\rrangle }
\]
is commutative and that the restrictions
\[
\begin{gathered}
\varphi^{-1}\{j\}
\to(\varphi^{(\sigma;u)})^{-1}\{j\}
\\
\varphi^{-1}\{-\infty,\infty\}\cap\langle m\rangle
\to(\varphi^{(\sigma;u)})^{-1}\{-\infty,\infty\}\cap\langle m\rangle
\end{gathered}
\]
are order-preserving or order-reversing according to $\varepsilon(x_j)$ for $1\le j\le n$ and $\varepsilon\theta(u)$ respectively.
This observation helps one prove the functoriality so as to check $W(H_0,H_1)$ is a presheaf over $\nabla$.

\begin{proposition}
\label{prop:interval-wr}
For every triangle as in \eqref{eq:grptri/Z/2}, the presheaf $W(H_0,H_1)$ over $\nabla$ forms a crossed interval group together with the degreewise group structure and the action on $\nabla(\llangle m\rrangle,\llangle n\rrangle)$ induced by the group homomorphism
\[
W(H_0,H_1)_n
\cong (\mathfrak S_n\ltimes H_1^{\times n})\times H_0
\xrightarrow{\text{proj.}\times\varepsilon\theta} \mathfrak S_n\times\mathbb Z/2\mathbb Z\ .
\]
\end{proposition}
\begin{proof}
By virtue of the above characterization of the permutation, one can easily prove the formula
\[
\begin{split}
\varphi^\ast((\tau;\vec y;v)\cdot (\sigma;\vec x;u))
&= (\varphi^{(\sigma;u)})^\ast(\tau;\vec y;v)\cdot \varphi^\ast(\sigma;\vec x;u) \\
&= (\varphi^{(\sigma;\vec x;u)})^\ast(\tau;\vec y;v)\cdot \varphi^\ast(\sigma;\vec x;u)\ .
\end{split}
\]
On the other hand, it is tedious but not difficult to show
\[
k^{((\varphi\psi)^{(\sigma;\vec x;u)})}_j
=\sum_{l\in(\varphi^{(\sigma;u)})^{-1}\{j\}} k^{(\varphi^\ast(\sigma;\vec x;u))}_j
= k^{(\varphi^{(\sigma;\vec x;u)}\psi^{\varphi^\ast(\sigma;\vec x;u)})}_j
\]
for each $j=1,\dots,n,\pm\infty$.
These show $W(H_0,H_1)$ is a crossed interval group.
\end{proof}

\begin{example}
\label{ex:crsint-symgrp}
Take $H_0=H_1=\ast$ the trivial group, one has
\[
W(\ast,\ast)_n
\cong\mathfrak S_n\ .
\]
Hence, there is a crossed interval group whose group on $\llangle n\rrangle$ is the permutation group $\mathfrak S_n$.
We will write $\mathfrak S=W(\ast,\ast)$.
Note that the coincidence of the notation with the operad of symmetric groups will be justified in \cref{sec:emb-ptdop}.
\end{example}

\begin{example}
\label{ex:crsgrp-hypoct}
The group $W(\ast,\mathbb Z/2\mathbb Z)_n$ is called the \emph{$n$-th hyperoctahedral group}, and we will denote the resulting crossed interval group by $\mathfrak H$.
It was proved in \cite{FL91} that the restriction of $\mathfrak H$ to the simplex category $\Delta$ through the functor
\[
\Delta\to\nabla
\ ;\quad [n]\cong\langle n+1\rangle\mapsto \llangle n+1\rrangle
\]
is the terminal object in the category $\mathbf{CrsGrp}_\Delta$ of crossed simplicial groups.
We call $\mathfrak H$ the hyperoctahedral crossed interval group.
\end{example}

In the interval case, we can compute the terminal crossed interval group explicitly.
Thanks to results of the paper \cite{Yoshidalimcolim}, this tell us how to compute limits and colimits in the category $\mathbf{CrsGrp}_\nabla$.

\begin{theorem}
\label{theo:crsint-termWeyl}
The crossed interval group $\mathfrak W_\nabla:= W(\mathbb Z/2\mathbb Z,\mathbb Z/2\mathbb Z)$ is a terminal object in the category $\mathbf{CrsGrp}_\nabla$.
Moreover, the following hold.
\begin{enumerate}[label={\rm(\arabic*)}]
  \item Colimits in $\mathbf{CrsGrp}_\nabla$ are computed degreewisely in the category $\mathbf{Grp}$ of groups.
  \item The forgetful functor $U:\mathbf{CrsGrp}_\nabla\to\mathbf{Set}_\nabla^{/\mathfrak W_\nabla}$ admits a left adjoint and creates arbitrary small limits and filtered colimits.
\end{enumerate}
\end{theorem}
\begin{proof}
The result follows from Example~3.17, Proposition~2.1, Corollary~2.2, and Proposition~2.3 in \cite{Yoshidalimcolim}.
\end{proof}

\begin{remark}
\label{rem:crsWeyl-subs}
There are inclusions of crossed interval groups
\[
\mathfrak S
\hookrightarrow \mathfrak H
\hookrightarrow \mathfrak W_\nabla\ .
\]
Actually, there are six crossed interval subgroups of $\mathfrak W_\nabla$, which are all found in the appendix of \cite{Yoshidalimcolim}.
\end{remark}

To establish an embedding of group operads into the category $\mathbf{CrsGrp}_\nabla$, we introduce another aspect of crossed interval groups.
Recall that group operads are defined as monoid objects in terms of a monoidal structure on the category $\mathbf{Op}^{/\mathfrak S}$.
On the other hand, \cref{theo:crsint-termWeyl} implies that we have a forgetful functor $\mathbf{CrsGrp}_\nabla\to\mathbf{Set}_\nabla^{/\mathfrak W_\nabla}$, where the codomain is the slice category of the category $\mathbf{Set}_\nabla$ of interval sets, i.e. presheaves over $\nabla$.
Actually, crossed interval groups are also monoid objects in $\mathbf{Set}_\nabla^{/\mathfrak W_\nabla}$.

\begin{definition}
Let $X$ and $Y$ be two interval sets, and suppose we are given a map $\rho:Y\to\mathfrak W_\nabla$ of interval sets.
Then, we define an interval set $X\rtimes Y$ as follows:
\begin{itemize}
  \item for each $n\in\mathbb N$, we set
\[
(X\rtimes Y)_n := X_n\times Y_n\ ;
\]
  \item for a morphism $\varphi:\llangle m\rrangle\to\llangle n\rrangle\in\nabla$, put
\[
\varphi^\ast:(X\rtimes Y)_n\to (X\rtimes Y_m)
\ ;\quad (x,y) \mapsto ((\varphi^{\rho(y)})^\ast(x),\varphi^\ast(y))\ .
\]
\end{itemize}
\end{definition}

One can check that $X\rtimes Y$ is actually an interval set with the data above.
Moreover, as in the case of operads, the assignment $(X,Y,\rho)\mapsto X\rtimes Y$ gives rise to a functor
\[
\rtimes:\mathbf{Set}_\nabla\times\mathbf{Set}^{/\mathfrak W_\nabla}_\nabla\to\mathbf{Set}_\nabla
\]
with, for $f:X\to X'\in\mathbf{Set}_\nabla$ and $g:Y\to Y'\in\mathbf{Set}^{/\mathfrak W_\nabla}_\nabla$,
\[
f\rtimes g:X\rtimes Y\to X'\rtimes Y'
\ ;\quad (x,y) \mapsto (f(x),g(y))\ .
\]
The following results were proved in Section~4 in \cite{Yoshidalimcolim}.

\begin{lemma}
The degreewise multiplication gives rise to a map
\[
\mathtt{mul}:\mathfrak W_\nabla\rtimes\mathfrak W_\nabla\to\mathfrak W_\nabla
\]
of crossed interval groups.
\end{lemma}

\begin{proposition}[Proposition~4.4 in \cite{Yoshidalimcolim}]
The composition of functors
\[
\begin{split}
\mathbf{Set}^{/\mathfrak W_\nabla}_\nabla\times\mathbf{Set}^{/\mathfrak W_\nabla}_\nabla
&\cong (\mathbf{Set}_\nabla\times\mathbf{Set}^{/\mathfrak W_\nabla}_\nabla)^{/(\mathfrak W_\nabla,\mathfrak W_\nabla)} \\
&\xrightarrow\rtimes \mathbf{Set}^{/\mathfrak W_\nabla\rtimes\mathfrak W_\nabla}_\nabla \\
&\xrightarrow{\mathtt{mul}_!} \mathbf{Set}^{/\mathfrak W_\nabla}_\nabla
\end{split}
\]
defines a monoidal structure on $\mathbf{Set}^{/\mathfrak W_\nabla}_\nabla$ with the unit object $\ast\to\mathfrak W_\nabla$ corresponding to the unit degreewisely.
\end{proposition}

We call a monoid object in $\mathbf{Set}^{/\mathfrak W}_\nabla$ with respect to the monoidal structure $\rtimes$ a \emph{crossed interval monoid}.
We denote by $\mathbf{CrsMon}_\nabla$ the category of crossed interval monoids and homomorphisms.
By the definition of the functor $\rtimes$, a crossed interval monoid consists of an interval set $G$ together with data
\begin{itemize}
  \item a monoid structure on $G_n$ for each $n\in\mathbb N$;
  \item a map $G\to\mathfrak W_\nabla$ of interval sets with $G_n\to(\mathfrak W_\nabla)_n$ being a monoid homomorphism, which endows $\nabla(\llangle m\rrangle,\llangle n\rrangle)$ with a left $G_n$-action;
\end{itemize}
satisfying certain conditions.
Comparing it with the definition of crossed interval groups, one obtains the following results.

\begin{lemma}[Lemma~4.8 in \cite{Yoshidalimcolim}]
\label{lem:crsgrp-monobj}
A crossed interval group $G$ determines a crossed interval monoid.
Moreover, it gives rise to a functor $\mathbf{CrsGrp}\to\mathbf{CrsMon}$.
\end{lemma}

\begin{proposition}[Proposition~4.13 in \cite{Yoshidalimcolim}]
\label{prop:crsgrp-grpobj}
The functor $\mathbf{CrsGrp}\to\mathbf{CrsMon}$ is fully faithful.
Moreover, a crossed monoid $M$ belongs to the essential image if and only if the monoid $M_n$ is a group for each $n\in\mathbb N$.
\end{proposition}

\begin{remark}
\label{rem:restrict-rtimes}
In the paper \cite{Yoshidalimcolim}, we considered more general situation; a monoidal structure $\rtimes_G$ was introduced on the category $\mathbf{Set}^{/G}_\nabla$ for arbitrary crossed interval group $G$.
The monoidal structure $\rtimes$ is recovered with $G=\mathfrak W_\nabla$.
It was proved that there is an equivalence
\[
\mathbf{CrsMon}^{/G}_\nabla
\cong\mathbf{Mon}(\mathbf{Set}^{/G}_\nabla,\rtimes_G)
\]
of categories, where the right hand side is the category of monoid objects.
In particular, if $G$ is a crossed interval subgroup of $\mathfrak W_\nabla$, so $\mathbf{Set}^{/G}_\nabla$ can be seen as a subcategory of $\mathbf{Set}^{/\mathfrak W_\nabla}_\nabla$, then the monoidal structure $\rtimes_G$ agrees with the restriction of $\rtimes$.
\end{remark}

\section{The embedding of pointed operads}
\label{sec:emb-ptdop}

In the previous sections, we prepared two notions of group operads and of crossed interval groups.
Comparing the definition of group operads with \cref{prop:crsgrp-grpobj}, the reader may have a feeling that they can be translated to one another.
The goal of this section is to make it clearer and to establish a fully faithful embedding $\mathbf{GrpOp}\hookrightarrow\mathbf{CrsGrp}_\nabla$.

We denote by $\mathbf{Op}^{\ast/}$ the category of pointed operads; i.e. the coslice category, or the under category, on the trivial operad $\ast$.
Since the set $\ast(n)$ is a singleton for each $n\in\mathbb N$, giving a map $\ast\to\mathcal X$ of operads is equivalent to giving a family $\{e_n\}_n$ of elements $e_n\in\mathcal X(n)$ satisfying
\begin{equation}
\label{eq:ptdoper-basept}
\gamma(e_n;e_{k_1},\dots,e_{k_n}) = e_{k_1+\dots+k_n}\ .
\end{equation}
We first define a functor $\Psi:\mathbf{Op}^{\ast/}\to\mathbf{Set}_\nabla$ as follows:
recall that morphisms $\llangle m\rrangle\to\llangle n\rrangle\in\nabla$ correspond in one to one to partitions $\vec k=(k_{-\infty},k_1,\dots,k_n,k_\infty)$ of $m$ into $(n+2)$ non-negative integers through the formula \eqref{eq:nablamorph-vect}.
To simplify the notation, for a pointed operad $\mathcal X$ with base points $e_n\in\mathcal X(n)$, if $\vec k^{(\varphi)}$ is the partition corresponding to $\varphi:\llangle m\rrangle\to\llangle n\rrangle\in\nabla$, then we write $e^\varphi_j:=e_{k^{(\varphi)}_j}$ for each $j\in\llangle n\rrangle$.
In this case, we define an interval set $\Psi(\mathcal X)$ by
\begin{itemize}
  \item for each $n\in\mathbb N$, $\Psi(\mathcal X)_n:=\mathcal X(n)$;
  \item for a morphism $\varphi:\llangle m\rrangle\to\llangle n\rrangle\in\nabla$, we set
\begin{equation}
\label{eq:defPsi}
\varphi^\ast:\Psi(\mathcal X)_n\to \Psi(\mathcal X)_m
\ ;\quad x\mapsto \gamma(e_3;e^\varphi_{-\infty},\gamma(x;e^\varphi_1,\dots,e^\varphi_n),e^\varphi_\infty) .
\end{equation}
\end{itemize}
Note that, by virtue of the equation \eqref{eq:ptdoper-basept}, for morphisms $\varphi:\llangle m\rrangle\to\llangle n\rrangle$ and $\psi:\llangle l\rrangle\to\llangle m\rrangle$ in $\nabla$, if $\varphi^{-1}\{j\}\setminus\{\pm\infty\}=\{i_1<\dots<i_{k^{(\varphi)}_j}\}\subset\llangle m\rrangle$, we have
\[
e^{\varphi\psi}_j = 
\begin{cases}
\gamma(e_2;e^\psi_{-\infty},\gamma(e^\varphi_{-\infty};e^\psi_{i_1},\dots,e^\psi_{i_{k^{(\varphi)}_j}})) & j=-\infty\ ,
\\
\gamma(e^\varphi_j;e^\psi_{i_1},\dots,e^\psi_{i_{k^{(\varphi)}_j}}) & 1\le j\le n\ ,
\\
\gamma(e_2;\gamma(e^\varphi_\infty;e^\psi_{i_1},\dots,e^\psi_{i_{k^{(\varphi)}_j}}),e^\psi_\infty) & j=\infty\ .
\end{cases}
\]
This and the associativity of the compositions in operads imply $\psi^\ast\varphi^\ast=(\varphi\psi)^\ast$ so that $\Psi(\mathcal X)$ is in fact an interval set.
On the other hand, if $f:\mathcal X\to\mathcal Y$ is a map of pointed operads, so we have $f(e_n)=e_n$, then the maps
\[
\Psi(f):\Psi(\mathcal X)_n\to\Psi(\mathcal Y)_n
\ ;\quad x\mapsto f(x)
\]
clearly define a map $\Psi(f):\Psi(\mathcal X)\to\Psi(\mathcal Y)$ of interval sets.
The functoriality is obvious so that we obtain a functor $\Psi:\mathbf{Op}^{\ast/}\to\mathbf{Set}_\nabla$.

\begin{example}
\label{ex:opsymgrp-ptd}
Note that every group operad is by definition pointed.
In particular, the operad $\mathfrak S$ canonically admits the map $\ast\to\mathfrak S$ corresponding to the unit of each $\mathfrak S(n)$.
The interval set $\Psi(\mathfrak S)$ is isomorphic to the interval set $\mathfrak S$ given in \cref{ex:crsint-symgrp}.
\end{example}

Thinking of $\mathfrak S$ as a pointed object in the category $\mathbf{Op}$, we can take the slice category $(\mathbf{Op}^{\ast/})^{/\mathfrak S}$ on $\mathfrak S$.
Then, one can observe that the monoidal structure $\rtimes$ on $\mathbf{Op}^{/\mathfrak S}$ lifts to $(\mathbf{Op}^{\ast/})^{/\mathfrak S}$.
Indeed, notice that there is an isomorphism $(\mathbf{Op}^{\ast/})^{/\mathfrak S}\cong(\mathbf{Op}^{/\mathfrak S})^{\ast/}$ so the monoidal structure $\rtimes$ induces a functor
\[
\begin{multlined}
(\mathbf{Op}^{\ast/})^{/\mathfrak S}\times(\mathbf{Op}^{\ast/})^{/\mathfrak S}
\cong (\mathbf{Op}^{/\mathfrak S})^{\ast/}\times (\mathbf{Op}^{/\mathfrak S})^{\ast/} \\
\xrightarrow\rtimes (\mathbf{Op}^{/\mathfrak S})^{(\ast\rtimes\ast)/}
\cong (\mathbf{Op}^{/\mathfrak S})^{\ast/}
\cong (\mathbf{Op}^{\ast/})^{/\mathfrak S}\ .
\end{multlined}
\]
It is easily verified this defines a monoidal structure on $(\mathbf{Op}^{\ast/})^{\mathfrak S/}$ so that the functor $(\mathbf{Op}^{\ast/})^{/\mathfrak S}\to\mathbf{Op}^{/\mathfrak S}$ is strictly monoidal.

\begin{lemma}
\label{lem:Psi-monoidal}
The functor
\[
\Psi^{\mathfrak S}:(\mathbf{Op}^{\ast/})^{/\mathfrak S}
\to \mathbf{Set}^{/\Psi(\mathfrak S)}_\nabla = \mathbf{Set}^{/\mathfrak S}_\nabla
\]
induced by the functor $\Psi$ defined above is strictly monoidal (see \cref{rem:restrict-rtimes} for the monoidal structure on $\mathbf{Set}^{/\mathfrak S}_\nabla$).
\end{lemma}
\begin{proof}
It is obvious that the functor $\Psi^{\mathfrak S}$ preserves the unit objects, namely $\Psi^{\mathfrak S}(\ast)=\ast$ with regard to the maps into $\Psi(\mathfrak S)=\mathfrak S$.
We have to show the equation $\Psi^{\mathfrak S}(\mathcal X\rtimes\mathcal Y)=\Psi^{\mathfrak S}(\mathcal X)\rtimes\Psi^{\mathfrak S}(\mathcal Y)$ for every pointed operads $\mathcal X$ and $\mathcal Y$ over $\mathfrak S$.
It clearly holds degreewisely, so it suffices to verify the structures of interval sets agree with each other.
Let $\varphi:\llangle m\rrangle\to\llangle n\rrangle\in\nabla$ be a morphism.
Say $\rho:\mathcal Y\to\mathfrak S$ is the structure map, then in the interval set $\Psi^{\mathfrak S}(\mathcal X\rtimes\mathcal Y)$, the induced map $\varphi^\ast:\Psi^{\mathfrak S}(\mathcal X\rtimes\mathcal Y)_n\to\Psi^{\mathfrak S}(\mathcal X\rtimes\mathcal Y)_m$ is given by
\begin{equation}
\label{eq:prf:Psi-monoidal:indmap}
\begin{split}
\varphi^\ast(x,y)
&= \gamma_{\mathcal X\rtimes\mathcal Y}(e_3;e^\varphi_{-\infty},\gamma_{\mathcal X\rtimes\mathcal Y}((x,y);e^\varphi_1,\dots,e^\varphi_n),e^\varphi_\infty) \\
& \begin{multlined}
=(\gamma_{\mathcal X}(e_3;e^\varphi_{-\infty},\gamma_{\mathcal X}(x;e^\varphi_{\rho(y)^{-1}(1)},\dots,e^\varphi_{\rho(y)^{-1}(n)}),e^\varphi_\infty), \\[1ex]
\gamma_{\mathcal Y}(e_3;e^\varphi_{-\infty},\gamma_{\mathcal Y}(y;e^\varphi_1,\dots,e^\varphi_n),e^\varphi_\infty))
\end{multlined}
\end{split}
\end{equation}
Note that for each $\sigma\in\mathfrak S_n$, we have
\[
e^{\varphi^\sigma}_{\pm\infty} = e^\varphi_{\pm\infty}
\ ,\quad e^{\varphi^\sigma}_j = e^\varphi_{\sigma^{-1}(j)}\ \text{for $1\le j\le n$}\ .
\]
Hence, the right hand side of \eqref{eq:prf:Psi-monoidal:indmap} can be written as $((\varphi^{\rho(y)})^\ast(x),\varphi^\ast(y))$, which is exactly the image of the pair $(x,y)$ under the map $\varphi^\ast:(\Psi^{\mathfrak S}(\mathcal X)\rtimes\Psi^{\mathfrak S}(\mathcal Y))_n\to(\Psi^{\mathfrak S}(\mathcal X)\rtimes\Psi^{\mathfrak S}(\mathcal Y))_m$.
It follows that $\Psi^{\mathfrak S}(\mathcal X\rtimes\mathcal Y)$ is identical to $\Psi^{\mathfrak S}(\mathcal X)\rtimes\Psi^{\mathfrak S}(\mathcal Y)$ as interval sets.
The structure maps into $\mathfrak S$ obviously coincide, so we obtain the result.
\end{proof}

\begin{theorem}
\label{theo:mongrp-emb}
The functor $\Psi^{\mathfrak S}$ induces a fully faithful functor
\[
\widehat\Psi:\mathbf{MonOp}\to\mathbf{CrsMon}_\nabla^{/\mathfrak S}.
\]
Moreover, it restricts to a right adjoint functor $\mathbf{GrpOp}\to\mathbf{CrsGrp}_\nabla^{/\mathfrak S}$.
\end{theorem}
\begin{proof}
Since the functor $\Psi^{\mathfrak S}$ is strictly monoidal as proved in \cref{lem:Psi-monoidal}, it induces a functor $\widehat\Psi$ between the categories of monoid objects.
More precisely, it sends a monoid operad $\mathcal X=(\mathcal X,\mathtt{mul},e)$ to the interval set $\Psi^{\mathfrak S}(\mathcal X)$ over $\mathfrak S$ together with the structure maps
\[
\begin{gathered}
\mathtt{mul}:\Psi^{\mathfrak S}(\mathcal X)\rtimes\Psi^{\mathfrak S}(\mathcal X)
= \Psi^{\mathfrak S}(\mathcal X\rtimes\mathcal X)
\xrightarrow{\Psi(\mathtt{mul})}\Psi^{\mathfrak S}(\mathcal X)
\\
e:\ast=\Psi^{\mathfrak S}(\ast)
\xrightarrow{\Psi(e)}\Psi^{\mathfrak S}(\mathcal X)\ .
\end{gathered}
\]
Hence, the monoid structure on each $\widehat\Psi(\mathcal X)_n$ coincides with that on $\mathcal X(n)$.
Clearly $\widehat\Psi$ is faithful, so we show it is also full.

Let $\mathcal X$ and $\mathcal Y$ be two monoid operads, and suppose $f:\widehat\Psi(\mathcal X)\to\Psi(\mathcal Y)$ be a map of crossed interval monoids.
To see $f$ comes from a map of monoid operads, it is enough to show that the maps
\[
f:\mathcal X(n)
=\Psi^{\mathfrak S}(\mathcal X)_n
\xrightarrow{f}\Psi^{\mathfrak S}(\mathcal Y)_n
=\mathcal Y(n)
\]
form a map of operads.
Since it preserves the unit elements, we have $f(e_n)=e_n$; in particular, $f$ preserves the identities $\mathrm{id}=e_1$ of operads.
If $x\in\mathcal X(n)=\Psi^{\mathfrak S}(\mathcal X)_n$ and $x_i\in\mathcal X(k_i)$ for $1\le i\le n$, the definition of monoid operads implies
\begin{equation}
\label{eq:prf:mongrp-emb}
\begin{split}
\gamma_{\mathcal X}(x;x_1,\dots,x_n)=\
&\gamma_{\mathcal X}(x;e_{k_1},\dots,e_{k_n}) \\
&\cdot\gamma_{\mathcal X}(e_3;e_0,x_1,e_{k_2+\dots+k_n}) \\
&\cdot\gamma_{\mathcal X}(e_3;e_{k_1},x_2,e_{k_3+\dots+k_n}) \\
&\qquad\vdots \\
&\cdot\gamma_{\mathcal X}(e_3;e_{k_1+\dots+k_{n-1}},x_n,e_0)\ .
\end{split}
\end{equation}
Taking the unique morphism $\mu:\llangle k_1+\dots+k_n\rrangle\to\llangle n\rrangle$ with $k^{(\mu)}_{\pm\infty}=0$ and $k^{(\mu)}_j=k_j$ for $j\in\langle n\rangle$ and the morphism $\rho_j:\llangle k_1+\dots+k_n\rrangle\to\llangle k_j\rrangle$ defined by
\[
\rho_j(i) :=
\begin{cases}
-\infty & i\le k_1+\dots+k_{j-1}\ , \\
i-(k_1+\dots+k_{j-1}) & k_1+\dots+k_{j-1}+1\le i\le k_1+\dots+k_j\ , \\
\infty & i\ge k_1+\dots+k_j+1\ ,
\end{cases}
\]
then, by the definition \eqref{eq:defPsi} of the functor $\Psi$, we can rewrite the formula \eqref{eq:prf:mongrp-emb} as follows:
\[
\gamma_{\mathcal X}(x;x_1,\dots,x_n)=\mu^\ast(x)\rho_1^\ast(x_1)\rho_2^\ast(x_2)\dots\rho_n^\ast(x_n)\ .
\]
The same formula also holds in $\mathcal Y$, and, since $f$ is a map of crossed interval monoids, we obtain
\[
\begin{split}
f(\gamma_{\mathcal X}(x;x_1,\dots,x_n))
&= f(\mu^\ast(x)\rho_1^\ast(x_1)\dots\rho_n^\ast(x_n)) \\
&= f(\mu^\ast(x))f(\rho_1^\ast(x_1))\dots f(\rho_n^\ast(x_n)) \\
&= \mu^\ast(f(x))\rho_1^\ast(f(x_1))\dots\rho_n^\ast(f(x_n)) \\
&= \gamma_{\mathcal Y}(f(x);f(x_1),\dots,f(x_n))\ .
\end{split}
\]
This implies that the maps $f:\mathcal X(n)\to\mathcal Y(n)$ form a map of monoid operads.

As for the last statement, it follows from \cref{prop:grpop->monop}, \cref{prop:crsgrp-grpobj}, \cref{theo:grpop-locpres}, \cref{theo:crsgrp-locpres}, \cref{theo:crsint-termWeyl}, and General Adjoint Functor Theorem (e.g. see Theorem~1.66 in \cite{AdamekRosicky1994}).
\end{proof}

Note that, by virtue of \cref{rem:crsWeyl-subs}, the inclusion $\mathfrak S\hookrightarrow\mathfrak W_\nabla$ of crossed interval group induces a fully faithful functor $\mathbf{CrsGrp}_\nabla^{/\mathfrak S}\hookrightarrow\mathbf{CrsGrp}_\nabla$.
Combining it with \cref{theo:mongrp-emb}, one obtains an embedding $\mathbf{GrpOp}\to\mathbf{CrsGrp}_\nabla$.
Hence, in the rest of the paper, we identify group operads with their images under this functor.
In particular, under this convention, we have $\mathfrak S=\widehat\Psi(\mathfrak S)$, which explains the coincidence of the notations.

\section{Operadic interval groups}
\label{sec:operintgrp}

In this section, we aim to determine the essential image of the embedding $\mathbf{GrpOp}\hookrightarrow\mathbf{CrsGrp}_\nabla$.
Since it is fully faithful, it will provide an alternative definition of group operads.
Furthermore, we will see that there is a larger class of crossed interval groups which are associated to operads.
This result suggests an extension of the notion of group operads.

First of all, we need to know about the category $\nabla$.

\begin{definition}
Let $\varphi:\llangle m\rrangle\to\llangle n\rrangle$ be a morphism in $\nabla$.
\begin{enumerate}[label={\rm(\arabic*)}]
  \item $\varphi$ is said to be \emph{active} if we have $\varphi:\varphi^{-1}\{\pm\infty\}\to\{\pm\infty\}$.
  \item $\varphi$ is said to be \emph{inert} if the restriction $\varphi:\varphi^{-1}\{1,\dots,n\}\to\{1,\dots,n\}$ is bijective.
\end{enumerate}
\end{definition}

\begin{remark}
In the paper \cite{Lur14}, Lurie considered the notions above for morphisms in the category $\mathbf{Fin}_\ast$ of pointed finite sets.
Actually, we have a functor
\[
\nabla\to\mathbf{Fin}_\ast
\ ;\quad \llangle n\rrangle\mapsto \llangle n\rrangle/\{\pm\infty\}\ .
\]
A morphism $\varphi:\llangle m\rrangle\to\llangle n\rrangle\in\nabla$ is active (resp. inert) if and only if so is its image in $\mathbf{Fin}_\ast$.
\end{remark}

The following results are easy to verify.

\begin{lemma}
\begin{enumerate}[label={\rm(\arabic*)}]
  \item Every morphism $\varphi:\llangle m\rrangle\to\llangle n\rrangle\in\nabla$ uniquely factors as $\varphi=\mu\rho$ with $\rho$ inert and $\mu$ active.
  \item Every inert morphism admits a unique section in $\nabla$.
\end{enumerate}
\end{lemma}

In the following arguments, inert morphisms play distinguished roles.
One reason is the definition of the functor $\widehat\Psi$; if $\mathcal G$ is a group operad, then the group $\mathcal G(k)$ admits embeddings into $\mathcal G(n)$ provided $k\le n$, namely the group homomorphisms of the form
\[
\mathcal G(k)\to\mathcal G(n)
\ ;\quad x\mapsto \gamma(e_3;e_p,x,e_q)
\]
with $n=k+p+q$.
In terms of the category $\nabla$, they are realized as the maps induced by inert morphisms $\rho:\llangle n\rrangle\to\llangle k\rrangle$.
For example if $\rho:\llangle n\rrangle\to\llangle k\rrangle$ is an inert morphism with the unique section $\delta:\llangle k\rrangle\to\llangle n\rrangle$, then the map $\rho^\ast:\mathfrak S(k)\to\mathfrak S(n)$ exhibits each permutation $\sigma$ on $\langle k\rangle$ as that on $\{\delta(1),\dots,\delta(k)\}\subset\langle n\rangle$.
This \emph{embeddability} is one of the characteristic properties of group operads among general crossed interval groups.

\begin{lemma}
\label{lem:inert-hom}
For a crossed interval group $G$, the following are equivalent:
\begin{enumerate}[label={\rm(\alph*)}]
  \item\label{cond:inert-hom:stab} the unique map $G\to\mathfrak W_\nabla$ of crossed interval groups factors through the hyperoctahedral crossed interval group $\mathfrak H\subset\mathfrak W_\nabla$.
  \item\label{cond:inert-hom:homo} for every inert morphism $\rho:\llangle n\rrangle\to\llangle k\rrangle$, and for every $x\in G_k$, we have $\rho^x=\rho$.
\end{enumerate}
In the case the conditions above are satisfied, each inert morphism $\rho:\llangle n\rrangle\to\llangle k\rrangle\in\nabla$ induces an injective group homomorphism $\rho^\ast:G_k\to G_l$.
\end{lemma}
\begin{proof}
Note that, for each $n\in\mathbb N$, the subgroup $\mathfrak H_n\subset(\mathfrak W_\nabla)_n$ consists of elements whose actions on inert morphisms are trivial.
Since the map $G\to\mathfrak W_\nabla$ respects the action on morphisms, this implies the conditions \ref{cond:inert-hom:stab} and \ref{cond:inert-hom:homo} are equivalent.

The last statement is directly follows from the definition of crossed groups (e.g. see Lemma~1.1 in \cite{Yoshidalimcolim}).
\end{proof}

\begin{remark}
The last statement in \cref{lem:inert-hom} is not equivalent to the two conditions.
For example, there is a crossed interval group $\mathfrak{Refl}$ which is the constant presheaf at $\mathbb Z/2\mathbb Z$ and the action on hom-sets ``reverses the order.''
This has non-trivial actions on inert morphisms while every morphism $\varphi:\llangle m\rrangle\to\llangle n\rrangle\in\nabla$ induces a group homomorphism $\varphi^\ast:\mathfrak{Refl}_n\to\mathfrak{Refl}_m$.
\end{remark}

The next property we discuss is the commutativity of elements with ``distinct supports.''
Suppose $\mathcal G$ is a group operad, and consider two embeddings
\[
\begin{gathered}
\eta:\mathcal G(k)\to\mathcal G(k+l)
\ ;\quad x\mapsto\gamma(e_2;x,e_l)\ ,
\\
\eta':\mathcal G(l)\to\mathcal G(k+l)
\ ;\quad y\mapsto\gamma(e_2;e_k,y)\ .
\end{gathered}
\]
Then, for each $x\in\mathcal G(k)$ and $y\in\mathcal G(l)$, we have $\eta(x)\eta'(y)=\eta'(y)\eta(x)$ in $G(k+l)$.
In other words, they induces an injective group homomorphism $\mathcal G(k)\times\mathcal G(l)\to\mathcal G(k+l)$.
To formulate this phenomenon in terms of the category $\nabla$, we introduce the following notion.

\begin{definition}
We say two morphisms $\varphi_1:\llangle n\rrangle\to\llangle k_1\rrangle$ and $\varphi_2:\llangle n\rrangle\to\llangle k_2\rrangle$ in $\nabla$ with the same domain are \emph{dissociated} if, for each morphism $\alpha:\llangle 1\rrangle\to\llangle n\rrangle\in\nabla$, either of the compositions $\varphi_1\alpha$ or $\varphi_2\alpha$ factors through the object $\llangle 0\rrangle\in\nabla$.
\end{definition}

Note that, in view of the identification $\nabla(\llangle 1\rrangle,\llangle n\rrangle)\cong\llangle n\rrangle$, two morphisms $\varphi_1:\llangle n\rrangle\to\llangle k_1\rrangle$ and $\varphi_2:\llangle n\rrangle\to\llangle k_2\rrangle$ in $\nabla$ are dissociated if and only if, for each $i\in\llangle n\rrangle$, either $\varphi_1(i)$ or $\varphi_2(i)$ is $\pm\infty$.
This observation leads to the result below.

\begin{lemma}
\label{lem:dissoc-comp}
Let $\varphi_i:\llangle n\rrangle\to\llangle k_i\rrangle$ ($i=1,2$) be morphisms in $\nabla$.
\begin{enumerate}[label={\rm(\arabic*)}]
  \item If $\varphi_1=\mu_1\rho_1$ and $\varphi_2=\mu_2\rho_2$ are the unique factorizations into inert morphisms followed by active morphisms, then $\varphi_1$ and $\varphi_2$ are dissociated if and only if $\rho_1$ and $\rho_2$ are so.
  \item For a morphism $\psi:\llangle m\rrangle\to\llangle\to\llangle n\rrangle\in\nabla$, the compositions $\varphi_1\psi$ and $\varphi_2\psi$ are dissociated provided so are $\varphi_1$ and $\varphi_2$.
\end{enumerate}
\end{lemma}

\begin{lemma}
\label{lem:dissoc-rstab}
Let $G$ be a crossed interval group satisfying the conditions in \cref{lem:inert-hom}, and suppose $\varphi_1:\llangle n\rrangle\to\llangle k_1\rrangle$ and $\varphi_2:\llangle n\rrangle\to\llangle k_2\rrangle$ are dissociated morphisms of $\nabla$.
Then, for each element $x\in G(k_1)$ and every morphism $\psi:\llangle m\rrangle\to\llangle n\rrangle\in\nabla$, we have
\[
\varphi_2\psi^{\varphi_1^\ast(x)}=\varphi_2\psi\ .
\]
\end{lemma}
\begin{proof}
Since the unique map $G\to\mathfrak W_\nabla$ factors through $G\to\mathfrak H$, and since it respects the actions on the morphisms of $\nabla$, it will suffice to verify the statement only in the case $G=\mathfrak H$.
Note that, for a morphism $\psi:\llangle m\rrangle\to\llangle n\rrangle\in\nabla$, and for an element $(\sigma;\vec\varepsilon)\in\mathfrak H_n$, the compositions $\varphi_2\psi^{(\sigma;\vec\varepsilon)}$ is, as a map, the composition
\[
\llangle m\rrangle
\xrightarrow{\psi^\ast(\sigma)^{-1}} \llangle m\rrangle
\xrightarrow{\psi} \llangle n\rrangle
\xrightarrow{\sigma} \llangle n\rrangle
\xrightarrow{\varphi_2} \llangle k_2\rrangle\ .
\]
If $(\sigma;\vec\varepsilon)$ belongs to the image of the map $\varphi_1^\ast:\mathfrak H_{k_1}\to\mathfrak H_n$, the dissociativity of $\varphi_1$ and $\varphi_2$ implies the permutation $\sigma$ is the identity except on exactly one of $\varphi_2^{-1}\{-\infty\}$ or $\varphi_2^{-1}\{\infty\}$.
Hence, we have $\varphi_2\circ\sigma=\varphi_2$ as maps.
In the same reason, by virtue of \cref{lem:dissoc-comp}, we obtain $\varphi_2\circ\psi\circ\psi^\ast(\sigma)^{-1}=\varphi_2\circ\psi$.
This completes the proof.
\end{proof}

Now, we formulate the ``operad-like'' crossed interval groups as below.

\begin{definition}
A crossed interval group $G$ is said to be \emph{operadic} if it satisfies the following:
\begin{enumerate}[label={\rm(\roman*)}]
  \item $G$ satisfies the equivalent conditions in \cref{lem:inert-hom};
  \item if $\rho_1:\llangle n\rrangle\to\llangle k_1\rrangle$ and $\rho_2:\llangle n\rrangle\to\llangle k_2\rrangle$ are dissociated inert morphisms in $\nabla$, then elements of the images $\rho_1^\ast(G_{k_1})$ commute with those of $\rho_2^\ast(G_{k_2})$; in other words, the commutator $[\rho_1^\ast(G_{k_1}),\rho_2^\ast(G_{k_2})]\subset G_n$ is trivial.
\end{enumerate}
\end{definition}

\begin{example}
\label{ex:grpop-operadic}
As expected, for every group operad $\mathcal G$, the crossed interval group $\widehat\Psi(\mathcal G)$ is operadic.
To see this, notice that if $\rho_1:\llangle n\rrangle\to\llangle k_1\rrangle$ and $\rho_2:\llangle n\rrangle\to\llangle k_2\rrangle$ are dissociated inert morphisms, then there are integers $k_{\pm\infty}$ and $l$ such that, for each $x_i\in\mathcal G(k_i)$,
\[
\rho_{i_1}^\ast(x_{i_1})
= \gamma(e_5;e_{-\infty},x_{i_1},e_l,e_{k_{i_2}},e_\infty)
\ ,\quad
\rho_{i_2}^\ast(x_{i_2})
= \gamma(e_5;e_{-\infty},e_{i_2},e_l,x_{i_2},e_\infty)
\]
for $\{i_1,i_2\}=\{1,2\}$.
Hence, the condition on group operads implies these elements commute with each other.
In view of \cref{lem:dissoc-comp}, it follows $\widehat\Psi(\mathcal G)$ is operadic.
\end{example}

\begin{example}
The crossed interval group $\mathfrak H$ is operadic.
This follows from the direct computation and that $\mathfrak S$ is operadic.
\end{example}

Operadic crossed interval groups are actually associated with operads.
To see it, we introduce some notions for simplicity; for a sequence $\vec k=(k_1,\dots,k_n)$ of non-negative integers, we put
\begin{itemize}
  \item $\mu_{\vec k}:\llangle k_1+\dots+k_n\rrangle\to\llangle n\rrangle\in\nabla$ to be the unique active morphism with $\#\mu_{\vec k}^{-1}\{j\}=k_j$ for each $1\le j\le n$; and
  \item $\rho^{(\vec k)}_j:\llangle k_1+\dots+k_n\rrangle\to\llangle k_j\rrangle\in\nabla$ to be the inert morphism given by
\[
\begin{split}
&\llangle k_1+\dots+k_n\rrangle \\
&\begin{multlined}
\cong \{-\infty,1,\dots,k_1+\dots+ k_{j-1}\}\star\langle k_j\rangle \\
\phantom{\cong}\star\{k_1+\dots+k_j+1,\dots,k_1+\dots+k_n,\infty\}
\end{multlined}
\\
&\xrightarrow{\mathtt{const}\star\mathrm{id}\star\mathtt{const}} \{-\infty\}\star\langle k_j\rangle\star\{\infty\} \\
&\cong\llangle k_j\rrangle\ .
\end{split}
\]
\end{itemize}
For an operadic crossed interval group $G$, we define an operad $\mathcal O_G$ as follows:
\begin{itemize}
  \item for each $n\in\mathbb N$, $\mathcal O_G(n):=G_n$;
  \item the composition
\[
\gamma:\mathcal O_G(n)\times\mathcal O_G(k_1)\times\dots\times\mathcal O_G(k_n)
\to\mathcal O_G(k_1+\dots+k_n)
\]
is given by
\[
\gamma(x;x_1,\dots,x_n)
:= \mu_{\vec k}^\ast(x)\cdot(\rho^{(\vec k)}_1)^\ast(x_1)\dots(\rho^{(\vec k)}_n)^\ast(x_n)\ ,
\]
where $\vec k=(k_1,\dots,k_n)$.
\end{itemize}
The associativity is seen as follows:
let $x\in G_n$, $x_i\in G_{k_i}$, and $x^{(i)}_s\in G_{k^{(i)}_s}$ for $1\le i\le n$ and $1\le s\le k_i$, and put $\vec l=(\sum_s k^{(1)}_s,\dots,\sum_s k^{(n)}_s)$.
Then, we have
\begin{equation}
\label{eq:OG-rgamma}
\begin{split}
&\gamma\bigl(x;\gamma(x_1;x^{(1)}_1,\dots,x^{(1)}_{k_1}),\dots,\gamma(x_n;x^{(n)}_1,\dots,x^{(1)}_{k_n})\bigr) \\
& \begin{multlined}
= \mu_{\vec l}^\ast(x)\cdot(\rho^{(\vec l)}_1)^\ast\bigl(\mu_{\vec k^{(1)}}^\ast(x_1)\cdot(\rho^{(\vec k^{(1)})}_1)^\ast(x^{(1)}_1)\dots(\rho^{(\vec k^{(1)})}_{k_1})^\ast(x^{(1)}_{k_1})\bigr) \\
\dots (\rho^{(\vec l)}_n)^\ast\bigl(\mu_{\vec k^{(n)}}^\ast(x_n)\cdot(\rho^{(\vec k^{(n)})}_1)^\ast(x^{(n)}_1)\dots(\rho^{(\vec k^{(n)})}_{k_n})^\ast(x^{(n)}_{k_n})\bigr)\ .
\end{multlined}
\end{split}
\end{equation}
Since $G$ is operadic, each $(\rho^{(\vec l)}_i)^\ast$ is a group homomorphism.
In addition, $\rho^{(\vec l)}_i$ and $\rho^{(\vec l)}_j$ are dissociated provided $i\neq j$, so by virtue of \cref{lem:dissoc-rstab}, the right hand side of \eqref{eq:OG-rgamma} is written as
\begin{multline}
\label{eq:OG-prelgamma}
\mu_{\vec l}^\ast(x)(\mu_{\vec k^{(1)}}\rho^{(\vec l)}_1)^\ast(x_1)\dots(\mu_{\vec k^{(n)}}\rho^{(\vec l)}_n)^\ast(x_n) \\
\cdot (\rho^{(\vec k^{(1)})}_1\rho^{(\vec l)}_1)^\ast(x^{(1)}_1)\dots(\rho^{(\vec k^{(1)})}_{k_1}\rho^{(\vec l)}_1)^\ast(x^{(1)}_{k_1})\dots(\rho^{(\vec k^{(n)})}_{k_n}\rho^{(\vec l)}_n)^\ast(x^{(n)}_{k_n})\ .
\end{multline}
Finally, using the formulas
\[
\begin{gathered}
\mu^{}_{\vec l}
= \mu^{}_{\vec k}\mu^{}_{\vec k^{(1)}\dots\vec k^{(n)}} \\[1ex]
\mu^{}_{\vec k^{(i)}}\rho^{(\vec l)}_i
= \rho^{(\vec k)}_i\mu^{}_{\vec k^{(1)}\dots\vec k^{(n)}} \\[1ex]
\rho^{(\vec k^{(i)})}_s\rho^{(\vec l)}_i
= \rho^{(\vec k^{(1)}\dots\vec k^{(n)})}_{k_1+\dots+k_{i-1}+s}
\end{gathered}
\]
with $\vec k^{(1)}\dots\vec k^{(n)}=(k^{(1)}_1,\dots,k^{(1)}_{k_1},\dots,k^{(n)}_{k_n})$, \eqref{eq:OG-prelgamma} is furthermore equal to
\[
\begin{split}
&\begin{multlined}
\mu_{\vec k^{(1)}\dots\vec k^{(n)}}^\ast\bigl(\mu_{\vec k}^\ast(x)(\rho^{(\vec k)}_1)^\ast(x_1)\dots(\rho^{(\vec k)}_n)^\ast(x_n)\bigr) \\[1ex]
\cdot (\rho^{(\vec k^{(1)}\dots\vec k^{(n)})}_1)^\ast(x^{(1)}_1)\dots(\rho^{(\vec k^{(1)}\dots\vec k^{(n)})}_{k_1})^\ast(x^{(1)}_{k_1})\dots(\rho^{(\vec k^{(1)}\dots\vec k^{(n)})}_{k_1+\dots+k_n})^\ast(x^{(n)}_{k_n})
\end{multlined} \\
&= \gamma(\gamma(x;x_1,\dots,x_n);x^{(1)}_1,\dots, x^{(1)}_{k_1},\dots,x^{(n)}_{k_n})\ .
\end{split}
\]
Clearly, the unit $e_1\in G_1=\mathcal O_G(1)$ behaves as the identity, so $\mathcal O_G$ is in fact an operad.

\begin{example}
\label{ex:O-grpop}
If $\mathcal G$ is a group operad, we have a strict identification
\[
\mathcal O_{\widehat\Psi(\mathcal G)} = \mathcal G\ .
\]
In particular, $\mathcal O_{\mathfrak S}=\mathfrak S$.
\end{example}

We denote by $\mathbf{CrsOpGrp}\subset\mathbf{CrsGrp}^{/\mathfrak H}_\nabla$ the full subcategory spanned by operadic crossed interval groups.
On the other hand, we define a category $\mathbf{Op}_{\mathrm{gr}}$ whose objects are operads $\mathcal O$ equipped with a group structure on each $\mathcal O(n)$ and whose morphisms are maps of operads which are group homomorphisms level-wisely.
Then, the assignment $G\mapsto\mathcal O_G$ clearly extends to a functor $\mathcal O_{(\blank)}:\mathbf{CrsOpGrp}\to\mathbf{Op}_{\mathrm{gr}}$.
By abuse of notation, we write $\mathfrak H=\mathcal O_{\mathfrak H}$, so we have the induced functor
\[
\mathcal O^{\mathfrak H}_{(\blank)}:\mathbf{CrsOpGrp}
\to\mathbf{Op}_{\mathrm{gr}}^{/\mathfrak H}\ .
\]

\begin{theorem}
\label{theo:operadic-emb}
The functor $\mathcal O^{\mathfrak H}_{(\blank)}$ is a fully faithful functor.
\end{theorem}
\begin{proof}
To obtain the result, it suffices to show that, for operadic crossed interval groups $G$ and $H$, a family $\{f:G_n\to H_n\}_{n\in\mathbb N}$ of group homomorphisms forms a map of crossed interval group if and only if it forms a map of operads.
This is verified almost identically to the first part of \cref{theo:mongrp-emb}.
\end{proof}

\begin{theorem}
\label{theo:operadic-local}
The subcategory $\mathbf{CrsOpGrp}\subset\mathbf{CrsGrp}_\nabla^{/\mathfrak H}$ is reflective; i.e. the inclusion functor admits a left adjoint.
\end{theorem}
\begin{proof}
For a crossed interval group $G$ over $\mathfrak H$, define $K(G)_n\subset G_n$ to be the subgroup generated by the commutators
\[
\begin{split}
[\rho_1^\ast(x_1),\rho_2^\ast(x_2)]
&= \rho_1^\ast(x_1)\rho_2^\ast(x_2)\rho_1^\ast(x_1)^{-1}\rho_2^\ast(x_2)^{-1} \\
&= \rho_1^\ast(x_1)\rho_2^\ast(x_2)\rho_1^\ast(x_1^{-1})\rho_2^\ast(x_2^{-1})
\end{split}
\]
for dissociated inert morphisms $\rho_i:\llangle n\rrangle\to\llangle k_i\rrangle$ and $x_i\in G_{k_i}$ for $i=1,2$.
We assert $K(G)=\{K(G)_n\}_n$ forms a crossed interval subgroup of $G$.
Indeed, for a morphism $\varphi:\llangle m\rrangle\to\llangle n\rrangle\in\nabla$, in view of \cref{lem:dissoc-rstab}, we have
\begin{equation}
\label{eq:prf:operadic-local:comm}
\begin{split}
&\varphi^\ast([\rho_1^\ast(x_1),\rho_2^\ast(x_2)]) \\
&\begin{multlined}
= (\rho_1\varphi^{\rho_2^\ast(x_2)\rho_1^\ast(x_1^{-1})\rho_2^\ast(x_2^{-1})})^\ast(x_1)(\rho_2\varphi^{\rho_1^\ast(x_1^{-1})\rho_2^\ast(x_2^{-1})})^\ast(x_2) \\
\cdot(\rho_1\varphi^{\rho_2^\ast(x_2^{-1})})^\ast(x_1^{-1})(\rho_2\varphi)^\ast(x_2^{-1})
\end{multlined} \\
&= (\rho_1\varphi^{\rho_1^\ast(x_1^{-1})\rho_2^\ast(x_2^{-1})})^\ast(x_1)(\rho_2\varphi^{\rho_2^\ast(x_2^{-1})})^\ast(x_2)(\rho_1\varphi)^\ast(x_1^{-1})(\rho_2\varphi)^\ast(x_2^{-1}) \\
&= ((\rho_1\varphi)^{x_1^{-1}})^\ast(x_1)((\rho_2\varphi)^{x_2^{-1}})^\ast(x_2)(\rho_1\varphi)^\ast(x_1^{-1})(\rho_2\varphi)^\ast(x_2^{-1}) \\
&= [((\rho_1\varphi)^{x_1^{-1}})^\ast(x_1),((\rho_2\varphi)^{x_2^{-1}})^\ast(x_2)]\ .
\end{split}
\end{equation}
Using \cref{lem:dissoc-comp}, one can see morphisms $(\rho_1\varphi)^{x_1^{-1}}$ and $(\rho_2\varphi)^{x_2^{-1}}$ are dissociated, so \eqref{eq:prf:operadic-local:comm} is one of generators of $K(G)_m$.
Hence, $K(G)$ is a crossed interval subgroup of $G$ by virtue of \cref{theo:crsint-termWeyl}.

Now, if $H$ is an operadic crossed interval group, then every map $f:G\to H$ of crossed interval groups sends the generators of $K(G)$ to the unit.
This in particular implies the map $K(G)\to\mathfrak H$ factors through the initial crossed interval group $\ast\subset\mathfrak H$, and the unique map $K(G)\to\ast$ of interval sets is actually a map of crossed interval groups.
We define a crossed interval group $\widetilde G$ by the following pushout square in $\mathbf{CrsGrp}_\nabla^{/\mathfrak H}$:
\[
\vcenter{
  \xymatrix{
    K(G) \ar[r] \ar@{^(->}[]+D+(0,-1);[d] & \ast \ar[d] \\
    G \ar[r] & \widetilde G }}
\ .
\]
It is obvious that $\widetilde G$ is operadic.
Moreover, the observation above shows that every map $G\to H$ of crossed interval groups uniquely factors through $G\to\widetilde G$ provided $H$ is operadic.
In other words, the assignment $G\mapsto\widetilde G$ gives the left adjoint of the inclusion $\mathbf{CrsOpGrp}\hookrightarrow\mathbf{CrsGrp}_\nabla^{/\mathfrak H}$, which is exactly the required result.
\end{proof}

\begin{example}
\label{ex:Delta->nabla}
Recall that we have a functor
\[
\mathfrak j:\Delta\to\nabla
\ ;\quad [n]\mapsto \{-\infty\}\star[n]\star\{\infty\}\cong\llangle n+1\rrangle
\]
from the simplex category $\Delta$.
By Theorem~5.14 in \cite{Yoshidalimcolim}, it induces a left adjoint functor
\[
\mathfrak j_\flat^{\mathfrak H}:\mathbf{CrsGrp}_\Delta\to\mathbf{CrsGrp}_\nabla^{/\mathfrak H}\ .
\]
According to the computations in \cite{Yoshidalimcolim} (precisely, Example~5.12 and 5.13), for a crossed simplicial group $G$, the crossed interval group $\mathfrak j^{\mathfrak H}_\flat G$ is described as follows:
for each $n\in\mathbb N$, the group $(\mathfrak j^{\mathfrak H}_\flat G)_n$ is the one generated by pairs $(x,\rho)$ of an inert morphism $\rho:\llangle n\rrangle\to\llangle k\rrangle$ and an element $x\in G_{k-1}$, with assuming $G_{-1}=\pi_0G$, which are subject to relation
\begin{equation}
\label{eq:jflat-rho}
(xy,\rho) \sim (x,\rho)(y,\rho)\ .
\end{equation}
In particular, $(\mathfrak j^{\mathfrak H}_\flat G)_n$ is the free product of copies of $G_{k-1}$ indexed by inert morphisms $\rho:\llangle n\rrangle\to\llangle k\rrangle$ with $k$ varying.
To obtain the ``operadification'' of $\mathfrak j^{\mathfrak H}_\flat G$, we only have to force the relation
\[
(x_1,\rho_1)(x_2,\rho_2) \sim (x_2,\rho_2)(x_1,\rho_1)
\]
for dissociated inert morphisms $\rho_i:\llangle n\rrangle\to\llangle k_i\rrangle$ and elements $x_i\in G_{k_i-1}$ in addition to \eqref{eq:jflat-rho}.
\end{example}

As seen in \cref{ex:grpop-operadic}, the embedding $\widehat\Psi:\mathbf{GrpOp}\to\mathbf{CrsGrp}_\nabla^{/\mathfrak H}$ factors through the subcategory $\mathbf{CrsOpGrp}\hookrightarrow\mathbf{CrsGrp}_\nabla^{/\mathfrak H}$.
To conclude the section, we compute the essential image of $\widehat\Psi$.
This is essentially achieved by interpreting the condition \eqref{eq:grpop-interchange} in terms of crossed interval groups.

\begin{definition}
A crossed interval group $G$ is said to be \emph{tame} if it satisfies the following condition:
for each sequence $\vec k=(k_1,\dots,k_n)$ and for each $x\in G_n$ and $x_i\in G_{k_i}$ for $1\le i\le n$, one has
\[
\mu_{\vec k}^\ast(x)\cdot(\rho^{(\vec k)}_i)^\ast(x_i)
= (\rho^{(x_\ast(\vec k))}_{x(i)})^\ast(x_i)\cdot\mu_{\vec k}^\ast(x)\ ,
\]
where we put $x_\ast(\vec k)=(k_{x^{-1}(1)},\dots,k_{x^{-1}(n)})$.
\end{definition}

\begin{theorem}
\label{theo:tame-embimg}
A crossed interval group $G$ belongs to the essential image of the functor $\widehat\Psi:\mathbf{GrpOp}\to\mathbf{CrsGrp}_\nabla$ if and only if it is operadic and tame and lies over $\mathfrak S\subset\mathfrak W_\nabla$.
\end{theorem}
\begin{proof}
If $G=\widehat\Psi(\mathcal G)$ for a group operad $\mathcal G$, then for $\vec k=(k_1,\dots,k_n)$, $x\in G_n$, and $x_i\in G_{k_i}$ for $1\le i\le n$, we have
\[
\begin{split}
\mu_{\vec k}^\ast(x)\cdot(\rho^{(\vec k)}_i)^\ast(x_i)
&= \gamma(x;e_{k_1},\dots,\overset{\substack{i\cr\smile}}{x_i},\dots,e_{k_n}) \\
&= \gamma(e_n;e_{k_{x^{-1}(1)}},\dots,\overset{\substack{x(i)\cr\smile}}{x_i},\dots,e_{k_{x^{-1}(n)}})\cdot\gamma(x;e_{k_1},\dots,e_{k_n}) \\
&= (\rho^{(x_\ast(\vec k))}_{x(i)})^\ast(x_i)\cdot\mu_{\vec k}^\ast(x)\ .
\end{split}
\]
Hence, $\widehat\Psi(\mathcal G)$ is tame as well as operadic.

Conversely, suppose $G$ is an operadic and tame crossed interval group over $\mathfrak S$.
We assert $\mathcal O_G$ is a group operad.
Indeed, in view of \cref{ex:O-grpop}, the operad $\mathcal O_G$ admits a canonical map $\mathcal O_G\to\mathcal O_{\mathfrak S}=\mathfrak S$ of operads which is level-wise group homomorphism.
In addition, since $G$ is tame, for $x,y\in G_n$ and $x_i,y_i\in G_{k_i}$, we have
\[
\begin{split}
\gamma(xy;x_1y_1,\dots,x_ny_n)
&= \mu_{\vec k}^\ast(xy)(\rho^{(\vec k)}_1)^\ast(x_1y_1)\dots(\rho^{(\vec k)}_n)^\ast(x_ny_n) \\
&\begin{multlined}
= \mu_{y_\ast(\vec k)}^\ast(x)\mu_{\vec k}^\ast(y) \\
\cdot(\rho^{(\vec k)}_1)^\ast(x_1)(\rho^{(\vec k)}_1)^\ast(y_1)\dots(\rho^{(\vec k)}_n)^\ast(x_n)(\rho^{(\vec k)}_n)^\ast(y_n)
\end{multlined} \\
&\begin{multlined}
= \mu_{y_\ast(\vec k)}^\ast(x)(\rho^{(y_\ast(\vec k))}_{y(1)})^\ast(x_1)\dots(\rho^{(y_\ast(\vec k))}_{y(n)})^\ast(x_n) \\
\cdot \mu_{\vec k}^\ast(y)(\rho^{(\vec k)}_1)^\ast(y_1)\dots(\rho^{(\vec k)}_n)^\ast(y_n)
\end{multlined} \\
&= \gamma(x;x_{y^{-1}(1)},\dots,x_{y^{-1}(n)})\gamma(y;y_1,\dots,y_n)\ .
\end{split}
\]
This implies $\mathcal O_G$ satisfies the condition \eqref{eq:grpop-interchange}, and it is a group operad.
Now, it is clear that $\widehat\Psi(\mathcal O_G)\cong G$, and this completes the proof.
\end{proof}

Similarly to the operadicity, for each crossed interval group $G$ over $\mathfrak S$, one can find a crossed interval subgroup $L(G)\subset G$ so that
\begin{enumerate}[label={\rm(\roman*)}]
  \item $L(G)$ is non-crossed; i.e. there is a map $L(G)\to\ast$ of crossed interval groups;
  \item $G\mapsto L(G)$ is functorial; i.e. every map $f:G\to H$ of crossed interval groups over $\mathfrak S$ restricts to $L(G)\to L(H)$;
  \item for each $n\in\mathbb N$, $L(G)_n\subset G_n$ contains all the elements of the form
\[
\mu_{\vec k}^\ast(x)\cdot(\rho^{(\vec k)}_i)^\ast(x_i)\cdot \mu_{\vec k}^\ast(x)^{-1}\cdot (\rho^{(x_\ast(\vec k))}_{x(i)})^\ast(x_i)^{-1}
\]
for $\vec k=(k_1,\dots,k_m)$ with $\sum k_i=n$, $x\in G_m$, and $x_i\in G_{k_i}$;
  \item $L(G)$ is trivial provided $G$ is tame.
\end{enumerate}
Then, the ``taming'' $G^{\mathsf t}$ of $G$ is obtained by the following pushout square in $\mathbf{CrsGrp}_\nabla^{/\mathfrak S}$
\[
\vcenter{
  \xymatrix{
    L(G) \ar[r] \ar@{^(->}[]+D+(0,-1);[d] & \ast \ar[d] \\
    G \ar[r] & G^{\mathsf t} }}
\quad,
\]
and it gives rise to a left adjoint to the inclusion of the full subcategory $\mathbf{CrsGrp}_\nabla^{\mathsf{tame}}\subset\mathbf{CrsGrp}_\nabla^{/\mathfrak S}$ spanned by tame crossed interval groups.
Moreover, since the operadification and the taming commute with each other, the latter is restricted so as to induce the left adjoint to the functor $\widehat\Psi$:
\[
(\blank)^{\mathsf t}:
\vcenter{
  \xymatrix{
    \mathbf{CrsOpGrp} \ar@/^.2pc/[]+R+(0,1);[r]+L+(0,1) \ar@{}[r]|-\perp & \mathbf{CrsOpGrp}^{\mathsf{tame}} \ar@/^.2pc/[]+L+(0,-1);[l]+R+(0,-1) }}
:\widehat\Psi\ ,
\]
where $\mathbf{CrsOpGrp}^{\mathsf{tame}}=\mathbf{CrsOpGrp}\cap\mathbf{CrsGrp}_\nabla^{\mathsf{tame}}$.
Note that, in view of \cref{theo:tame-embimg}, the functor $\widehat\Psi$ induces an equivalence (actually an isomorphism) $\mathbf{GrpOp}\simeq\mathbf{CrsOpGrp}^{\mathsf{tame}}$.
Hence, we obtain an explicit description of the left adjoint to $\widehat\Psi$, which has been proved to exist in \cref{theo:mongrp-emb}.

\begin{remark}
It was proved in Theorem~3.5 in \cite{Gur15} that the category $\mathbf{GrpOp}$ is locally presentable.
The observation above gives us an alternative proof of this fact:
in view of \cref{theo:mongrp-emb}, we may regard $\mathbf{GrpOp}$ as a reflective subcategory of $\mathbf{CrsGrp}^{/\mathfrak S}_\nabla$.
It is verified that operadic crossed interval groups and tame ones are closed under filtered colimits respectively.
Then, $\mathbf{GrpOp}$ is locally presentable thanks to Corollary to Theorem~2.48 in \cite{AdamekRosicky1994}.
\end{remark}

\section{Associative algebras}
\label{sec:assalg}

In this final section, we give an application of the results established in the previous sections.

We begin with the following observation.
Let $\mathcal C$ be a monoidal category.
Then, the category of monoid objects of $\mathcal C$ is equivalent to the category $\mathbf{MultCat}(\ast,\mathcal C^\otimes)$ of multifunctors from the terminal operad to the multicategory $\mathcal C^\otimes$ associated to $\mathcal C$.
We write
\[
\mathbf{Alg}(\mathcal C):=\mathbf{MultCat}(\ast,\mathcal C^\otimes)\ .
\]
The assignment gives rise to a $2$-functor
\[
\mathbf{Alg}(\blank):
\mathbf{MonCat}
\to\mathbf{Cat}\ .
\]
It was shown in Section~\expandafter\uppercase\expandafter{\romannumeral7\relax}.5 in \cite{McL98} that the $2$-functor is \emph{represented} by the category $\widetilde\Delta$ of finite ordinals and order-preserving maps with the join $\star$ as the monoidal structure.
Indeed, as easily checked, the following data defines a multifunctor $M:\ast\to\widetilde\Delta^\otimes$:
\begin{itemize}
  \item for the unique object $\ast$ of $\ast$, we set $M(\ast):=\langle 1\rangle$;
  \item for the unique operations $\mu_n\in\ast(n)$, we set
\[
M(\mu_n)
\in\widetilde\Delta^\otimes(\langle 1\rangle\dots\langle 1\rangle;\langle 1\rangle)
= \widetilde\Delta(\langle 1\rangle\star\dots\star\langle 1\rangle,\langle 1\rangle)
= \widetilde\Delta(\langle n\rangle,\langle 1\rangle)
\]
to be the unique morphism to the terminal object $\langle 1\rangle$.
\end{itemize}
Then, the precomposition with $M$ gives rise to a functor
\[
\mathbf{MonCat}(\widetilde\Delta,\mathcal C)
\xrightarrow{(\blank)^\otimes}\mathbf{MultCat}(\widetilde\Delta^\otimes,\mathcal C^\otimes)
\xrightarrow{M^\ast}\mathbf{Alg}(\mathcal C)
\]
for each monoidal category $\mathcal C$, which is claimed to be an equivalence.

On the other hand, if $\mathcal G$ is a group operad, then we can consider the composition
\begin{equation}
\label{eq:symmon-algfunc}
\mathbf{MonCat}_{\mathcal G}
\xrightarrow{\mathit{forget}} \mathbf{MonCat}
\xrightarrow{\mathbf{Alg}(\blank)} \mathbf{Cat}\ .
\end{equation}
We discuss the representability of this $2$-functor.
We first have to construct a candidate of the representing object.
In view of \cref{theo:mongrp-emb}, we regard the category $\mathbf{CrsGrp}$ as a full subcategory of $\mathbf{CrsGrp}_{\nabla}^{/\mathfrak S}$, so we may identify $\mathcal G$ with its image $\Psi^{\mathfrak S}(\mathcal G)$ in $\mathbf{CrsGrp}_{\nabla}^{/\mathfrak S}$.
On the other hand, we regard the category $\widetilde\Delta$ as a subcategory of $\nabla$ with all the objects and morphisms $\varphi:\llangle m\rrangle\to\llangle n\rrangle$ with $\varphi^{-1}\{\pm\infty\}=\{\pm\infty\}$, and put $\mathfrak J:\widetilde\Delta\to\nabla$ the inclusion.
Note that crossed $\widetilde\Delta$-groups are usually called \emph{augmented crossed simplicial groups} since the canonical embedding $\Delta\to\widetilde\Delta$ induces a fully faithful functor $\mathbf{CrsGrp}_{/\Delta}\to\mathbf{CrsGrp}_{/\widetilde\Delta}$ whose essential image consists of $G$ with $G_0$ trivial in the degree in $\widetilde\Delta\subset\nabla$ (not in $\Delta$).
Hence, pulling back along $\mathfrak J$, we obtain an augmented crossed simplicial group $\mathfrak J^\natural\mathcal G$ by virtue of Theorem~5.14 in \cite{Yoshidalimcolim} so as to form the total category $\widetilde\Delta_{\mathfrak J^\natural\mathcal G}$ (see \cref{rem:totcat}).

\begin{lemma}
\label{lem:augDelta-totmon}
Let $\mathcal G$ be a group operad.
Then, the monoidal structure on $\widetilde\Delta$ extends to the total category $\widetilde\Delta_{\mathfrak J^\natural\mathcal G}$ so that it is $\mathcal G$-symmetric.
\end{lemma}
\begin{proof}
Since the inclusion $\widetilde\Delta\to\widetilde\Delta_{\mathfrak J^\natural\mathcal G}$ is bijective on objects, we have to extends the monoidal product on morphisms.
Recall that morphisms of $\widetilde\Delta_{\mathfrak J^\natural\mathcal G}$ are of the form $(\varphi,x):\langle m\rangle\to\langle n\rangle$ with $\varphi:\langle m\rangle\to\langle n\rangle\in\widetilde\Delta$ and $x\in\mathcal G(m)$.
For morphisms $(\varphi_i,x_i):\langle m_i\rangle\to\langle n_i\rangle$ for $1=1,2$, we set
\[
(\varphi_1,x_1)\star(\varphi_2,x_2)
:=\bigl(\varphi_1\star\varphi_2,\gamma(e_2;x_1,x_2)\bigr)\ .
\]
The functoriality is obvious, and the strict associativity follows from those of the join and the composition operation in the operad $\mathcal G$.

To introduce a $\mathcal G$-symmetric structure, note that we have
\[
(\widetilde\Delta^\otimes_{\mathfrak J^\natural\mathcal G}\rtimes\mathcal G)(\langle k_1\rangle\dots\langle k_n\rangle;\langle m\rangle)
= \widetilde\Delta_{\mathfrak J^\natural\mathcal G}(\langle k_1+\dots+k_n\rangle,\langle m\rangle)\times\mathcal G(n).
\]
We then consider the map
\begin{equation}
\label{eq:prf:augDelta-totmon:Gsym}
\begin{array}{ccc}
  (\widetilde\Delta^\otimes_{\mathfrak J^\natural\mathcal G}\rtimes\mathcal G)(\langle k_1\rangle\dots\langle k_n\rangle;\langle m\rangle) &\to& \widetilde\Delta^\otimes_{\mathfrak J^\natural\mathcal G}(\langle k_1\rangle\dots\langle k_n\rangle;\langle m\rangle) \\[1ex]
  ((\varphi,x),u) &\mapsto& \bigl(\varphi,x\cdot\gamma(u;e_{k_1},\dots,e_{k_n}\bigr)\ .
\end{array}
\end{equation}
Using the description of the crossed interval group $\mathfrak S$ in \cref{ex:crsint-symgrp}, one can check the following formula:
\[
\begin{split}
&\bigl(\varphi,x\cdot\gamma(u;e_{k_1},\dots,e_{k_n})\bigr)
\circ\bigl(\varphi_1\star\dots\star\varphi_n,\gamma(e_n;x_1,\dots,x_n)\bigr) \\
&=
\begin{multlined}[t]
\bigl(\varphi\circ(\varphi_{u^{-1}(1)}\star\dots\star\varphi_{u^{-1}(n)})^x, \\
(\varphi_{u^{-1}(1)}\star\dots\star\varphi_{u^{-1}(n)})^\ast(x)\cdot\gamma(u;x_1,\dots,x_n)\bigr)\ ,
\end{multlined}
\end{split}
\]
which implies \eqref{eq:prf:augDelta-totmon:Gsym} actually defines an identity-on-object multifunctor $\widetilde\Delta_{\mathfrak J^\natural\mathcal G}^\otimes\rtimes\mathcal G\to\widetilde\Delta_{\mathfrak J^\natural\mathcal G}^\otimes$.
It is obvious that it exhibits $\widetilde\Delta_{\mathfrak J^\natural\mathcal G}$ as a $\mathcal G$-symmetric monoidal category.
\end{proof}

We assert that the $\mathcal G$-symmetric monoidal category $\widetilde\Delta^\otimes_{\mathfrak J^\natural\mathcal G}$ in fact represents the $2$-functor \eqref{eq:symmon-algfunc}.
To see this, notice that, since the $2$-category $\mathbf{MultCat}_{\mathcal G}$ is the category of $2$-algebras over the $2$-monad $(\blank)\rtimes\mathcal G$, we have an isomorphism
\[
\mathbf{MultCat}(\mathcal M,\mathcal N)
\cong\mathbf{MultCat}_{\mathcal G}(\mathcal M\rtimes\mathcal G,\mathcal N)
\]
of categories for each $\mathcal M\in\mathbf{MultCat}$ and $\mathcal N\in\mathbf{MultCat}_{\mathcal G}$.
Hence, for $\mathcal G$-symmetric monoidal category $\mathcal C$, we have a canonical isomorphism
\[
\mathbf{Alg}(\mathcal C)
= \mathbf{MultCat}(\ast,\mathcal C^\otimes)
\cong\mathbf{MultCat}_{\mathcal G}(\mathcal G,\mathcal C^\otimes)\ ,
\]
here we use the isomorphism $\ast\rtimes\mathcal G\cong\mathcal G$.
We define a multifunctor $M_{\mathcal G}:\mathcal G\to\widetilde\Delta_{\mathfrak J^\natural\mathcal G}$ as follows:
\begin{itemize}
  \item for the unique object $\ast$ of $\mathcal G$, set $M_{\mathcal G}(\ast):=\langle 1\rangle$;
  \item for each $x\in\mathcal G(n)$, we put $M_{\mathcal G}(x):=(\mu_n,x)\in\widetilde\Delta_{\mathfrak J^\natural\mathcal G}(\langle n\rangle,\langle 1\rangle)$, where $\mu_n:\langle n\rangle\to\langle 1\rangle\in\widetilde\Delta$ is the unique map into the terminal object in $\widetilde\Delta$.
\end{itemize}
It is straightforward from the description \eqref{eq:prf:augDelta-totmon:Gsym} that $M_{\mathcal G}$ is even $\mathcal G$-symmetric.

\begin{proposition}
\label{prop:Alg-Gsymrep}
Let $\mathcal G$ be a group operad.
Then, for every $\mathcal G$-symmetric monoidal category $\mathcal C$, the functor
\begin{equation}
\label{eq:Alg-Gsymrep:ind}
\begin{multlined}
\mathbf{MonCat}_{\mathcal G}(\widetilde\Delta_{\mathfrak J^\natural\mathcal G},\mathcal C)
\xrightarrow{(\blank)^\otimes}\mathbf{MultCat}_{\mathcal G}(\widetilde\Delta_{\mathfrak J^\natural\mathcal G}^\otimes,\mathcal C^\otimes) \\
\xrightarrow{M_{\mathcal G}^\ast} \mathbf{MultCat}_{\mathcal G}(\mathcal G,\mathcal C^\otimes)
\cong\mathbf{Alg}(\mathcal C)\ ,
\end{multlined}
\end{equation}
where the second functor is induced by the precomposition with $M_{\mathcal G}:\mathcal G\to\widetilde\Delta^\otimes_{\mathfrak J^\natural\mathcal G}$ is an equivalence of categories.
\end{proposition}
\begin{proof}
We actually construct the inverse to the composition, say $M^\otimes_{\mathcal G}$, of the first two functors in \eqref{eq:Alg-Gsymrep:ind}.
We use the unbiased convention for monoidal categories described in Chapter 3 in \cite{Lei04}.

Notice first that every morphism $\langle m\rangle\to\langle n\rangle\in\widetilde\Delta$ can be uniquely written in the form $\mu_{k_1}\star\dots\star\mu_{k_n}$ for a sequence $\vec k=(k_1,\dots,k_n)$ with $k_1+\dots+k_n=m$ (cf. \cref{eq:nablamorph-vect}), where $\mu_k:\langle k\rangle\to\langle 1\rangle$ is the unique map into the terminal object $\langle 1\rangle$ in $\widetilde\Delta$.
For a $\mathcal G$-symmetric multifunctor $F:\mathcal G\to\mathcal C^\otimes$, we define a functor $F_\otimes:\widetilde\Delta_{\mathfrak J^\natural\mathcal G}\to\mathcal C$ as follows:
\begin{itemize}
  \item for each $n\in\mathbb N$, we set $F_\otimes(\langle n\rangle):=F(\ast)^{\otimes n}$, where $\ast$ is the unique object of the operad $\mathcal G$;
  \item for each morphism $(\varphi,x):\langle m\rangle\to\langle n\rangle\in\widetilde\Delta_{\mathfrak J^\natural\mathcal G}$, say $\varphi=\mu_{k_1}\star\dots\star\mu_{k_n}$, we define a morphism $F_\otimes(\varphi,x):F(\ast)^{\otimes m}\to F(\ast)^{\otimes n}$ to be the composition
\[
F(\ast)^{\otimes m}
\xrightarrow{\Theta^x} F(\ast)^{\otimes m}
\cong F(\ast)^{\otimes k_1}\otimes\dots\otimes F(\ast)^{\otimes k_n}
\xrightarrow{F(e_{k_1})\otimes\dots\otimes F(e_{k_1})} F(\ast)^{\otimes n}
\ ,
\]
where $\Theta^x$ is the natural transformation defined in \eqref{eq:Gsym-moncat}.
\end{itemize}
The naturality of $\Theta^x$ implies that, if we have a sequence $(l_1,\dots,l_m)$, the square below commutes in $\mathcal C$:
\[
\vcenter{
  \xymatrix@C=10em{
    F(\ast)^{\otimes l_1}\otimes\dots\otimes F(\ast)^{\otimes l_m} \ar[d]_{\Theta^x} \ar[r]^-{F(e_{l_1})\otimes\dots\otimes F(e_{l_m})} & F(\ast)^{\otimes m} \ar[d]^{\Theta^x} \\
    F(\ast)^{\otimes l_{x^{-1}(l_1)}}\otimes\dots\otimes F(\ast)^{\otimes l_{x^{-1}(m)}} \ar[r]^-{F(e^{}_{l_{x^{-1}(1)}})\otimes\dots\otimes F(e^{}_{l_{x^{-1}(m)}})} & F(\ast)^{\otimes m} }}
\quad.
\]
Since the left vertical arrow agrees with the morphism $\Theta^{\gamma(x;e_{l_1},\dots,e_{l_m})}$ under the isomorphisms
\[
\begin{split}
F(\ast)^{\otimes l_1}\otimes\dots\otimes F(\ast)^{\otimes l_m}
&\cong F(\ast)^{\otimes(l_1+\dots+l_m)} \\
&\cong F(\ast)^{\otimes l_{x^{-1}(l_1)}}\otimes\dots\otimes F(\ast)^{\otimes l_{x^{-1}(m)}}\ ,
\end{split}
\]
the functoriality of $F_\otimes$ follows.
In addition, it is straightforward from the definition that $F_\otimes$ is monoidal, with the comparison isomorphism
\[
\begin{split}
F(\langle k_1\rangle)\otimes\dots\otimes F(\langle k_n\rangle)
&= F(\ast)^{\otimes k_1}\otimes\dots\otimes F(\ast)^{\otimes k_n} \\
&\cong F(\ast)^{\otimes(k_1+\dots+k_n)} \\
&= F(\langle k_1\rangle\star\dots\star\langle k_n\rangle)\ ,
\end{split}
\]
and $\mathcal G$-symmetric.
On the other hand, note that a multinatural transformation $\alpha:F\to G:\mathcal G\to\mathcal C^\otimes$ consists of a morphism $\alpha:F(\ast)\to G(\ast)\in\mathcal C$ such that, for each $k\in\mathbb N$, the square below is commutative:
\[
\vcenter{
  \xymatrix{
    F(\ast)^{\otimes k} \ar[d]_{F(e_k)} \ar[r]^{\alpha^{\otimes k}} & G(\ast)^{\otimes k} \ar[d]^{G(e_k)} \\
    F(\ast) \ar[r]^\alpha & G(\ast) }}
\quad.
\]
It immediately follows that, setting $(\alpha_\otimes)_{\langle k\rangle}:=\alpha^{\otimes k}$, we get a natural transformation $\alpha_\otimes:F_\otimes\to G_\otimes$.
Hence, the construction above defines a functor
\[
(\blank)_\otimes:
\mathbf{MultCat}_{\mathcal G}(\mathcal G,\mathcal C^\otimes)
\to \mathbf{MonCat}_{\mathcal G}(\widetilde\Delta_{\mathfrak J^\natural\mathcal G},\mathcal C)\ .
\]

The composition $M^\otimes_{\mathcal G}\circ(\blank)_\otimes$ is clearly identified with the identity functor on $\mathbf{MultCat}_{\mathcal G}(\mathcal G,\mathcal C^\otimes)$.
On the other hand, if $F:\widetilde\Delta_{\mathfrak J^\natural\mathcal G}\to\mathcal C$ is a $\mathcal G$-symmetric monoidal functor, it is equipped with an isomorphism
\[
\lambda_n:(M^\otimes_{\mathcal G}(F))_\otimes(\langle n\rangle)
= M^\otimes_{\mathcal G}(F)(\ast)^{\otimes n}
= F(\langle 1\rangle)^{\otimes n}
\xrightarrow\cong F(\langle n\rangle)
\]
for each $n\in\mathbb N$.
We assert $\lambda=\{\lambda_n\}_n$ forms a natural isomorphism $M^\otimes_{\mathcal G}(\blank)_\otimes\cong\mathrm{Id}$.
Indeed, for each sequence $k_1,\dots,k_n$ of non-negative integers, the coherence diagram for $\lambda$ guarantees the following diagram to commute:
\begin{equation}
\label{eq:prf:Alg-Gsymrep}
\vcenter{
  \xymatrix{
    F(\langle 1\rangle)^{\otimes {k_1+\dots+k_n}} \ar[r]^-\cong \ar[d]_{\lambda_{k_1+\dots+k_n}} & F(\langle 1\rangle)^{\otimes k_1}\otimes\dots\otimes F(\langle 1\rangle)^{\otimes k_n} \ar[d]^{\lambda_{k_1}\otimes\dots\otimes\lambda_{k_n}} \\
    F(\langle k_1+\dots+k_n\rangle) \ar[d]_{F(\mu_{k_1}\star\dots\star\mu_{k_n})} & F(\langle k_1\rangle)\otimes\dots\otimes F(\langle k_n\rangle) \ar[l]_-{\lambda_n} \ar[d]^{F(\mu_{k_1})\otimes\dots\otimes F(\mu_{k_n})} \\
    F(\langle n\rangle) & F(\langle 1\rangle)^{\otimes n} \ar[l]_{\lambda_n} }}
\quad,
\end{equation}
where $\mu_k:\langle k\rangle\to\langle 1\rangle$ is the unique morphism to the terminal object.
Notice that, according to the description of the multifunctor $F^\otimes:\widetilde\Delta_{\mathfrak J^\natural\mathcal G}^\otimes\to\mathcal C^\otimes$ in the unbiased convention, the composition of the top arrow and the right vertical arrows exactly gives the morphism
\[
M^\otimes_{\mathcal G}(F)_\otimes(\mu_{k_1}\star\dots\star\mu_{k_n}):M^\otimes_{\mathcal G}(F)_\otimes(\langle k_1+\dots+k_n\rangle)\to M^\otimes_{\mathcal G}(F)_\otimes(\langle n\rangle)\ .
\]
Hence, it follows from the commutativity of \eqref{eq:prf:Alg-Gsymrep} that $\lambda$ is a natural isomorphism.
Moreover, the upper half square in \eqref{eq:prf:Alg-Gsymrep} is precisely the coherence diagram for $\lambda$ to be monoidal.
In other words, we obtain an isomorphism $\lambda:M^\otimes_{\mathcal G}(F)_\otimes\cong F$ in the category $\mathbf{MonCat}_{\mathcal G}(\widetilde\Delta_{\mathfrak J^\natural\mathcal G},\mathcal C)$.
If $\alpha:F\to G:\widetilde\Delta_{\mathfrak J^\natural\mathcal G}\to\mathcal C$ is a monoidal natural transformation, then the coherence of $\alpha$ gives us the equation
\[
\lambda\circ M^\otimes_{\mathcal G}(\alpha)_\otimes
= \alpha\circ\lambda\ ,
\]
so that $\lambda$ actually defines a natural isomorphism $M^\otimes_{\mathcal G}(\blank)_\otimes\cong\mathrm{Id}$, which exhibits the functor $(\blank)_\otimes$ as an inverse to $M^\otimes_{\mathcal G}$.
\end{proof}

To conclude the paper, we mention the relation of the category $\widetilde\Delta_{\mathfrak J^\natural\mathcal G}$ to the Hochschild homologies for algebras.
For this, we need to recall the \emph{paracyclic category} $\Lambda_\infty$, which has the following description due to \cite{GetzlerJones1993} and \cite{Elmendorf1993} (see also Section~3 of \cite{Nistor1990}):
\begin{itemize}
  \item the objects are natural numbers $n\in\mathbb N$;
  \item the hom-set $\Lambda_\infty(m,n)$ consists of order-preserving maps $f:\mathbb Z\to\mathbb Z$ (with respect to the standard linear order on the integers) such that for each $i\in\mathbb Z$, we have
\[
f(i+m+1)=f(i)+n+1\ ;
\]
  \item the composition is the obvious one.
\end{itemize}
As pointed out by Fiedorowicz and Loday, $\Lambda_\infty$ is the total category of a crossed simplicial group $\mathbfcal Z$.
Indeed, there is a canonical faithful and bijective-on-object functor $\Delta\to\Lambda_\infty$, and the unique factorization in $\Lambda_\infty$ exhibits $\mathbfcal Z=\{\operatorname{Aut}_{\Lambda_\infty}(n)\}_n$ as a crossed simplicial group.
It is easily verified that $\mathbfcal Z_n\cong\mathbb Z$ with the simplicial degree.
The following result is due to Elmendorf.

\begin{proposition}[Proposition~1 in \cite{Elmendorf1993}]
\label{prop:paracyc-dual}
There is an identity-on-object functor $\overbar{(\blank)}:\Lambda_\infty^\opposite\to\Lambda_\infty$ such that, for $f\in\lambda_\infty(m,n)$,
\[
\overbar f:\mathbb Z\to\mathbb Z
\ ;\quad j \mapsto \min\{i\in\mathbb Z\mid j\le f(i)\}\ .
\]
\end{proposition}

\begin{example}
\label{ex:paracyc-dual-simp}
Let $\varphi:[m]\to[n]\in\Delta$ be a map in the simplex category.
Under the identification of it with the image in $\Lambda_\infty$, the composition
\[
\widehat\varphi:
\{0,\dots,n\}
\hookrightarrow\mathbb Z
\xrightarrow{\overbar\varphi}\mathbb Z
\twoheadrightarrow\mathbb Z/(m+1)\mathbb Z
\cong \{0,\dots,m\}
\]
is described as follows:
\[
\widehat\varphi^{-1}\{i\} =
\begin{cases}
\ \{0,1,\dots,\varphi(0),\varphi(m)+1,\dots,n\} & i=0\ , \\[1ex]
\ \{\varphi(i-1),\varphi(i-1)+1,\dots,\varphi(i)\} & 1\le i\le m\ .
\end{cases}
\]
Hence, seeing $\overbar\varphi$ as a morphism in the total category $\Delta_{\mathbfcal Z}$, we can write it as a pair
\[
\overbar\varphi
= \bigl(\mu_{\varphi(0)+n-\varphi(m)}\star\mu_{\varphi(1)-\varphi(0)}\star\dots\star\mu_{\varphi(m)-\varphi(m-1)},\tau^{n-\varphi(m)}_n\bigr),
\]
where $\star$ is the join of totally ordered sets, $\mu_k:[k-1]\to[0]$ is the unique morphism to the terminal object in $\Delta$, and $\tau_n\in\mathbfcal Z_n$ is the element corresponding to $1\in\mathbb Z$ under the canonical isomorphism $\mathbfcal Z_n\cong \mathbb Z$.
\end{example}

Now, suppose $\mathcal C$ is a $\mathcal G$-symmetric monoidal abelian category; i.e. it is abelian and equipped with a $\mathcal G$-symmetric monoidal structure so that the functor
\[
\mathord\otimes_n:\mathcal C^{\times n}\to\mathcal C
\]
is right exact in each variable for each $n\in\mathbb N$.
We denote by $\mathbf{Ch}(\mathcal C)$ the category of chain complexes in $\mathcal C$.
For a monoid object $A\in\mathcal C$, the Hochschild complex $C_\bullet(A)\in\mathbf{Ch}(\mathcal C)$ (with coefficient $A$) has the following construction, which is based on the one given in \cite{NikolausScholze2017}.
By virtue of \cref{prop:Alg-Gsymrep}, $A$ gives rise to a $\mathcal G$-symmetric monoidal functor $A_\otimes:\widetilde\Delta_{\mathfrak J^\natural\mathcal G}\to\mathcal C$.
On the other hand, the embedding described in Example~5.5 in \cite{Yoshidalimcolim} enables us to see $\mathbfcal Z$ as an augmented crossed simplicial group, so we can also form the total category $\widetilde\Delta_{\mathbfcal Z}$.
Note that the embedding $\Lambda_\infty\cong\Delta_{\mathbfcal Z}\hookrightarrow\widetilde\Delta_{\mathbfcal Z}$ identifies $\Lambda_\infty$ with the full subcategory of $\widetilde\Delta_{\mathbfcal Z}$ spanned by all but the initial object.
Hence, choosing a map $\mathbfcal Z\to\mathfrak J^\natural\mathcal G$ of augmented crossed simplicial groups, we obtain a \emph{paracyclic object} in $\mathcal C$:
\[
A^\circledcirc_\bullet:
\Lambda_\infty^\opposite
\cong\Lambda_\infty
\hookrightarrow\widetilde\Delta_{\mathbfcal Z}
\to\widetilde\Delta_{\mathfrak J^\natural\mathcal G}
\xrightarrow{A_\otimes} \mathcal C\ ,
\]
where the first isomorphism is the one given in \cref{prop:paracyc-dual}.
Then, the computations in \cref{ex:paracyc-dual-simp} shows that the Hochschild complex $C_\bullet(A)\in\mathbf{Ch}(\mathcal C)$ is isomorphic to the associated chain complex of the simplicial object $\left.A^\circledcirc\right|_{\Delta^\opposite}$.

\begin{remark}
In the above construction, we chose a map $\mathbfcal Z\to\mathfrak J^\natural\mathcal G$ of augmented crossed simplicial groups.
Note that, since $\mathbfcal Z$ is connected as a simplicial set, so the computation in Example~5.12 in \cite{Yoshidalimcolim} shows that for every augmented crossed simplicial group $G$, maps $\mathbfcal Z\to G$ of crossed simplicial groups correspond in one-to-one to those of augmented ones.
It follows that the Hochschild chain $C_\bullet(A)$ is actually indexed by the following isomorphic sets:
\[
\mathbf{CrsGrp}_{\Delta}(\mathbfcal Z,\mathfrak J^\natural\mathcal G)
\cong\mathbf{CrsGrp}_{\widetilde\Delta}(\mathbfcal Z,\mathfrak J^\natural\mathcal G)
\cong\mathbf{CrsGrp}_\nabla(\mathfrak J_\flat\mathbfcal Z,\mathcal G)\ .
\]
\end{remark}

\begin{example}
Take $\mathcal G=\mathcal B$ the group operad of braid groups.
In this case, there is the following canonical pullback square
\[
\vcenter{
  \xymatrix{
    \mathbfcal Z \ar[r] \ar[d] \ar@{}[dr]|(.4)\pbcorner & \mathfrak J^\natural\mathcal B \ar[d] \\
    \mathbfcal C \ar@{^(->}[]+R+(1,0);[r] & \mathfrak J^\natural\mathfrak S }}
\]
in the category $\mathbf{CrsGrp}_\Delta$, where $\mathbfcal C\subset\mathfrak J^\natural\mathfrak S$ is the crossed simplicial subgroup of cyclic groups.
Hence, there is a canonical choice of a Hochschild chain for a monoid object in a braided monoidal abelian category.
\end{example}


\bibliographystyle{plain}
\bibliography{mybiblio}

\begin{thebibliography}{10}

\bibitem{AdamekRosicky1994}
J.~Ad{\'a}mek and J.~Rosick{\'y}.
\newblock {\em Locally Presentable and Accessible Categories}, volume 189 of
  {\em London Mathematical Society Lecture Note Series}.
\newblock Cambridge University Press, 1994.

\bibitem{BataninMarkl2014}
M.~Batanin and M.~Markl.
\newblock Crossed interval groups and operations on the hochschild cohomology.
\newblock {\em Journal of Noncommutative Geometry}, pages 655--693, 2014.

\bibitem{CG13}
A.~S. Corner and N.~Gurski.
\newblock Operads with general groups of equivariance, and some $2$-categorical
  aspects of operads in cat.
\newblock arXiv:1312.5910, 2013.

\bibitem{Elmendorf1993}
A.~D. Elmendorf.
\newblock A simple formula for cyclic duality.
\newblock {\em Proceedings of the American Mathematical Society},
  118(3):709--711, 1993.

\bibitem{FL91}
Z.~Fiedorowicz and J-L Loday.
\newblock Crossed simplicial groups and their associated homology.
\newblock {\em Transactions of the American Mathematical Society},
  326(1):57--87, 1991.

\bibitem{GetzlerJones1993}
E.~Getzler and J.~D.~S. Jones.
\newblock The cyclic homology of crossed product algebras.
\newblock {\em Journal f\"ur die Reine und Angewandte Mathematik. [Crelle's
  Journal]}, 445:161--174, 1993.

\bibitem{Gur15}
N.~Gurski.
\newblock Operads, tensor products, and the categorical borel construction.
\newblock arXiv:1508.04050, 2015.

\bibitem{Hermida2000}
C.~Hermida.
\newblock Representable multicategories.
\newblock {\em Advances in Mathematics}, 151(2):164--225, 2000.

\bibitem{Kra87}
R.~Krasauskas.
\newblock Skew-simplicial groups.
\newblock {\em Lithuanian Mathematical Journal}, 27(1):47--54, 1987.
\newblock Translated from Litovski\u\i{} Matematichenski\u\i{} Sbornik
  (Lietuvos Matematikos Rinkinys), 27(1):89--99.

\bibitem{Lei04}
T.~Leinster.
\newblock {\em Higher operads, higher categories}.
\newblock Number 298 in London Mathematical Society Lecture Note Series.
  Cambridge University Press, Cambridge, 2004.

\bibitem{Lur14}
J.~Lurie.
\newblock Higher algebra.
\newblock see author's webpage, September 2014.

\bibitem{McL98}
S.~MacLane.
\newblock {\em Categories for the Working Mathematician}.
\newblock Number~5 in Graduate Texts in Mathematics. Springer-Verlag, second
  ed. edition, 1998.

\bibitem{NikolausScholze2017}
T.~Nikolaus and P.~Scholze.
\newblock On topological cyclic homology.
\newblock arXiv:1707.01799, 2017.

\bibitem{Nistor1990}
V.~Nistor.
\newblock Group cohomology and the cyclic cohomology of crossed products.
\newblock {\em Inventiones mathematicae}, 99(1):411--424, December 1990.

\bibitem{Wah01}
N.~Wahl.
\newblock {\em Ribbon braids and related operads}.
\newblock PhD thesis, University of Oxford, 2001.

\bibitem{Yoshidalimcolim}
J.~Yoshida.
\newblock Limits and colimits of crossed groups.
\newblock in preparation.

\bibitem{Zha11}
W.~Zhang.
\newblock Group operads and homotopy theory.
\newblock arXiv:1111.7090, part of the Ph.D. thesis, 2011.

\end{thebibliography}
\end{document}